\documentclass[12pt,reqno]{amsart}
\usepackage[margin=1in]{geometry}
\usepackage[hidelinks]{hyperref}
\usepackage{breakurl}

\usepackage{multicol}
\usepackage{parcolumns}

\usepackage{algorithm}
\usepackage{algpseudocode}

\usepackage{float}

\usepackage{ytableau}
\ytableausetup{mathmode, baseline, boxsize=1.4em}
\usepackage{transparent}

\def\sqpt{\raisebox{0.15pt}{\scalebox{0.6}{$\square$}}}
\def\fsqpt{\raisebox{0.15pt}{\scalebox{0.6}{$\blacksquare$}}}
\def\greycirc{\raisebox{.5pt}{\transparent{0.4}\tikz \draw[above=2cm] [fill=black,draw=black]  circle (2pt);}}

\usepackage{amsmath,amsthm,amssymb}

\usepackage{tikz}
\usepackage{tikz-cd}
\usepackage{graphicx}
\usepackage{caption}
\usepackage{subcaption}
\usepackage{float}
\usepackage[export]{adjustbox}
\usepackage{bbold}

\usepackage{booktabs}  
\usepackage{mathtools} 
\usepackage{units} 
\usepackage{ytableau} 
\usepackage{biblatex}  
\usepackage{nicematrix} 

\usepackage{ytableau}

\theoremstyle{plain}
\newtheorem{theorem}{Theorem}
\newtheorem{proposition}{Proposition}[section]
\newtheorem{corollary}[theorem]{Corollary}
\newtheorem{lemma}[proposition]{Lemma}

\theoremstyle{definition}
\newtheorem{definition}[proposition]{Definition}

\theoremstyle{definition}
\newtheorem{notation}[proposition]{Notation}

\newtheorem{example}[proposition]{Example}

\newtheorem{question}[proposition]{Question}
\newtheorem{conjecture}[proposition]{Conjecture}

\title[Weighted Borel Generators]{Weighted Borel Generators}
\author{Seth Ireland}
\address{Dept.~of Mathematics, Colorado State University, Fort Collins, CO, USA}

\email{seth.ireland@colostate.edu}

\begin{document}

\newcommand{\B}[1]{Borel(#1)}

\ytableausetup{centertableaux}

\begin{abstract}
Strongly stable ideals are a class of monomial ideals which correspond to generic initial ideals in characteristic zero and can be described completely by their Borel generators, a subset of the minimal monomial generators of the ideal. In \cite{francisco2011borel}, Francisco, Mermin, and Schweig develop formulas for the Hilbert series and Betti numbers of strongly stable ideals in terms of their Borel generators. In this work, a \textit{specialization} of strongly stable ideals is presented which further restricts the subset of relevant generators. A choice of weight vector $w\in\mathbb{N}_{> 0}^n$ restricts the set of strongly stable ideals to a subset designated as $w$-stable ideals. This restriction further compresses the Borel generators to a subset termed the \textit{weighted Borel generators} of the ideal. A new Macaulay2 package \verb|wStableIdeals.m2| has been developed alongside this paper and segments of code support computations within.
\end{abstract}

\maketitle

\section{Introduction}

Consider an invertible matrix $\alpha=(a_{ij})\in GL_n(\mathbb{K})$ for some field $\mathbb{K}$ and the action on the polynomial ring $\mathbb{K}[x_1,\dots,x_n]$ by
\begin{equation*}
    (x_1,\dots,x_n)\mapsto \Big(\sum_{i=1}^n a_{i1}x_i,\dots,\sum_{i=1}^n a_{in}x_i\Big).
\end{equation*}
Given an ideal $I\subset\mathbb{K}[x_1,\dots,x_n]$, a choice of monomial order $<$ (which respects the variable order $x_1>x_2>\cdots > x_n$) determines the initial ideal $in_<(I)$. If $\alpha\in GL_n(\mathbb{K})$ is \textit{generic}, then $in_<(\alpha\cdot I)$ is called the \textit{generic initial ideal}, denoted $gin_<(I)$. There is a filtration of subgroups
\begin{equation*}
    T_n(\mathbb{K})\subset B_n(\mathbb{K})\subset GL_n(\mathbb{K})
\end{equation*}
where $T_n(\mathbb{K})$ is the set of diagonal matrices and $B_n(\mathbb{K})$ is the set of upper-triangular matrices called the \textit{Borel subgroup}. By considering the ideals which are fixed by each of these subgroups, we get a filtration going the other direction. All monomial ideals are fixed by $T_n(\mathbb{K})$. Only powers of the maximal ideal $(x_1,\dots,x_n)$ are fixed by $GL_n(\mathbb{K})$. The intermediate class fixed by $B_n(\mathbb{K})$ are the Borel-fixed ideals and correspond exactly to generic initial ideals. In 1974, Galligo showed that generic initial ideals are Borel-fixed in characteristic zero \cite{galligo}. Bayer and Stillman extended this result to arbitrary characteristic in 1987 \cite{bayer1987theorem}. In 2004, Conca showed that if an ideal is Borel-fixed, then it is its own generic inital ideal \cite{conca2004koszul}. In summary, for any characteristic, the generic initial ideals are exactly the Borel-fixed ideals.

In characteristic zero, the Borel-fixed ideals (and therefore the generic initial ideals) correspond exactly to a class of monomial ideals called strongly stable ideals. Strongly stable ideals (also called Borel ideals) are a class of monomial ideals which are closed under a simple combinatorial operation known as a \textit{Borel move} \cite{herzog2011monomial}. In positive characteristic, Borel-fixed ideals are strongly stable if the characteristic is larger than any exponent appearing in any monomial generator of the ideal. This correspondence allows one to study invariants of generic initial ideals by taking advantage of the combinatorial properties of strongly stable ideals.

An arbitrary monomial ideal $I$ can be described by its set of minimal monomial generators, $G(I)$. When $I$ is strongly stable, it can be described by a subset $Bgens(I)\subseteq G(I)$, whose elements are called \textit{Borel generators}. In \cite{francisco2011borel}, Francisco, Mermin, and Schweig show how many invariants of strongly stable ideals can be computed directly from their Borel generators.

In this work, we \textit{specialize} the notion of a strongly stable ideal to that of a $w$-stable ideal where $w\in\mathbb{N}^n_{>0}$ is a monotone decreasing tuple of positive integers, called a weight vector. For any fixed weight vector $w$, the set of $w$-stable ideals is a restriction of the the set of strongly stable ideals. This restriction further compresses the minimal generating set from to a subset we call the \textit{weighted Borel generators}, denoted $Bgens_w(I)$. In this paper, we show how several invariants for $w$-stable ideals can be computed from their weighted Borel generators. 

Computations on arbitrary $w$-stable ideals can be decomposed into computations on \textit{principal} $w$-stable ideals (Propositions \ref{prop:basic1}, \ref{prop:basic2}, and \ref{prop:basic3}). The combinatorial structure of principal $w$-stable ideals is explored in Sections \ref{sec:principaltrees} and \ref{sec:catalandiagrams}, where we construct an $n$-ary tree and Catalan diagram (respectively) corresponding to each principal $w$-stable ideal. We show that these Catalan diagrams contain all the information needed to compute the Hilbert series (Theorems \ref{thm:principalhs} and \ref{thm:catalan1}) and graded Betti numbers (Theorems \ref{thm:catalan2} and \ref{cor:principal_poincare}) of principal $w$-stable ideals. In Section \ref{sec:principalcone}, we show how to determine when a given ideal is principally $w$-stable and for which weight vectors (Theorem \ref{principalConeThm}). In Section \ref{sec:apps}, we investigate a class of strongly stable ideals that appear to be well-behaved as $w$-stable ideals. We end with a conjecture linking non-homogeneous generic Gr\"obner fans and $w$-stable ideals (Conjecture \ref{conj1}).

Several other generalizations of strongly stable ideals have emerged over the last few decades. In his 1994 thesis, Pardue gives a combinatorial definition of $p$-Borel ideals and shows that they correspond exactly to Borel-fixed ideals in characteristic $p>0$ \cite{pardue1994nonstandard}. In 2013, Francisco, Mermin, and Schweig define $Q$-Borel ideals for a poset $Q$ on the variables $x_1,\dots,x_n$ \cite{francisco2013generalizing}. Recently (2021), DiPasquale and Nezhad define $L$-Borel ideals as a generalization of strongly stable ideals for which Borel moves need only be satisfied for a subset of the variables \cite{dipasquale2021koszul}. In 2002, Gasharov, Hibi, and Peeva define (strongly) $a$-stable ideals for an integer sequence $a=(a_1,\dots,a_n)$ as the image of (strongly) stable ideals in a polynomial ring with powers of the variables restricted by $a$ \cite{gasharov2002resolutions}. All four of these frameworks restrict the set of Borel moves under which an ideal must be closed. Contrast this with \cite{dalzotto2005non}, where Dalzotto and Sbarra define \textit{weighted strongly stable ideal} by \textit{weakening} the Borel move. They show that weighted strongly stable ideals correspond exactly to the upper triangular automorphisms of a weighted polynomial algebra $\mathbb{K}[x_1,\dots,x_n]$. This generalization is natural in the sense that it generalizes the correspondence between Borel-fixed ideals and strongly stable ideals to a correspondence between (upper triangular) graded automorphisms and weighted strongly stable ideals. However, as noted in their paper and in \cite{cox2005lectures}, this definition is unwieldy because \textit{every} monomial ideal is weighted strongly stable for a suitable choice of weight vector. This makes their weighted strongly stable ideals more difficult to work with than (classical) strongly stable ideals. For example, determining an analogous resolution to the Eliahou-Kervaire resolution \cite{eliahou1990minimal} (for arbitrary weight vector) would be equivalent to determining a minimal resolution for \textit{any} monomial ideal.

Throughout the paper, we use (version 1.0 of) the Macaulay2 package wStableIdeals.m2 to help with computations. The code for this package can be found at
\begin{center}
    \href{https://github.com/slireland94/wStableIdeals}{https://github.com/slireland94/wStableIdeals}
\end{center}
and is expected to appear in a future release of Macaulay2.

\section{Strongly Stable Ideals and w-Stable Ideals}\label{sec:background}

\subsection{Strongly Stable Ideals}
In this section, we set notation, recall definitions, and state basic facts about strongly stable ideals. We will work over the polynomial ring $R=K[y_1,\dots,y_n]$ with the variable order $y_1>y_2>\cdots > y_n$.

\begin{definition}
    Let $m\in R$ be a monomial. A \textit{Borel move} on $m$ is an operation that sends $m$ to a monomial
    \begin{equation*}
        m\cdot\frac{y_{i_1}}{y_{j_1}}\cdots\frac{y_{i_s}}{y_{j_s}},
    \end{equation*}
    where $i_t<j_t$ for all $t$ and $y_{j_1}\cdots y_{j_s}$ divides $m$.
\end{definition}

Given a monomial ideal $J\subset R$, we will use $G(J)$ to denote the minimal set of monomial generators. A monomial ideal $J\subset R$ is called \textit{strongly stable} if it is fixed by Borel moves. Explicitly, for all $m\in J$, if $u\in R$ is a monomial obtained by a Borel move on $m$, then $u\in J$.

We will now recall the definition of the Borel closure of a set of monomials in $R$ following notation in \cite{francisco2011borel}.
\begin{definition}
    Let $T=\{m_1,\dots,m_s\}\subset R$ be a set of monomials. Define $\B T$ to be the smallest strongly stable ideal containing $T$.
\end{definition}

\begin{definition}\label{def:borelorder}
    Factor $m_1=y_{i_1}y_{i_2}\cdots y_{i_r}$ and $m_2=y_{j_1}y_{j_2}\cdots y_{j_s}$ with $r\geq s$. If $i_k\leq j_k$ for all $k\leq s$, then we say that $m_1$ precedes $m_2$ in the \textit{Borel order} and write $m_1\succ m_2$.
\end{definition}

\subsection{w-Stable Ideals}\label{subsec:wstablebasics}

Let $\mathbb{N}^n_{>0}$ denote the set of tuples of positive integers. We will call any monotone decreasing tuple of positive integers $w=(w_1,\dots,w_n)\in\mathbb{N}^n_{>0}$ a \textit{weight vector}. Now let $S=\mathbb{K}[x_1,\dots,x_n]$, fix a weight vector $(w_1,\dots,w_n)$, and construct the ring homomorphism
\begin{align*}
    \psi:S&\rightarrow R\\
    x_i^{}&\mapsto y_i^{w_i}.
\end{align*}

Now, we define $w$-stable ideals in the polynomial ring $S$. Note that $R$ has variable order $y_1>y_2>\cdots > y_n$. In the last subsection, we defined the Borel closure of a set of monomials in $R$. Similarly, we can compute a \textit{weighted Borel closure} for a set of monomials $A\subseteq S$ by
\begin{equation*}
    \overline{A}:=\psi^{-1}(\B{\psi(A)}).
\end{equation*}
Monomial ideals in $S$ which are fixed by this closure operation are called $w$-stable ideals. When $A=\{m\}$ is a singleton, we will use $\overline{m}$ to denote the \textit{principal} $w$-stable ideal generated by $m$.

\begin{definition}
    A monomial ideal $I\subseteq S$ is \textit{$w$-stable} if $I=\overline{I}$.
\end{definition}

\begin{example}
    Let $w=(2,1)$ be the weight vector for $S=\mathbb{K}[x_1,x_2]$ and let $I=(x_1,x_2^2)\subset S$. Then $\psi(x_1)=y_1^2$ and $\psi(x_2)=y_2$, so $\psi(I)=(y_1^2,y_2^2)\subset R$. Then
    \begin{align*}
        \psi^{-1}(\B{\psi(I)}) &= \psi^{-1}((y_1^2,y_1y_2,y_2^2))\\
        &=I,
    \end{align*}
    so $I$ is $(2,1)-stable$. Note that the monomial $y_1y_2$ has empty preimage in $S$.

We can use the Macaulay2 package wStableIdeals to check our work.
\newpage
{\footnotesize
\begin{verbatim}


i1 : loadPackage "wStableIdeals"

o1 = wStableIdeals

o1 : Package

i2 : K = QQ

o2 = QQ

o2 : Ring

i3 : S = K[x_1,x_2]

o3 = S

o3 : PolynomialRing

i4 : I = ideal(x_1,x_2^2)

                 2
o4 = ideal (x , x )
             1   2

o4 : Ideal of S

i5 : borelClosure(I,Weights=>{2,1})

                 2
o5 = ideal (x , x )
             1   2

o5 : Ideal of S

i6 : iswStable(I,w)

o6 = true


\end{verbatim}}
We used the \verb|borelClosure| method to compute $\overline{I}=\psi^{-1}(\B{\psi(I)})$. Since the output is equal to $I$, $I$ is $(2,1)$-stable. The \verb|iswStable| method computes $\overline{I}$ and checks whether $\overline{I}=I$.
\end{example}

By default, weights in \verb|wStableIdeals.m2| are set to the vector of all ones. When $w=\mathbb{1}$ is the weight vector of all ones ($S$ has standard grading), we recover the classic definition of strongly stable ideals. For any fixed weight vector $w$, the set of $w$-stable ideals is a subset of the set of strongly stable ideals as the next proposition shows.

\begin{proposition}
    Fix a weight vector $w$. If $I\subset S$ is $w$-stable, then $I$ is strongly stable.
\end{proposition}

\begin{proof}
    Let $m\in I$ with $x_j\vert m$ and $i<j$. We need to show that $m\frac{x_i}{x_j}\in I$. Notice that
    \begin{align*}
        \psi(m\frac{x_i}{x_j}) &= \psi(m)\frac{y_i^{w_i}}{y_j^{w_j}}\\
        &= \psi(m)\frac{y_i^{w_j}}{y_j^{w_j}}y_i^{w_i-w_j}
    \end{align*}
    Since $x_j\vert m\in I$, we have $y_j^{w_j}\vert\psi(m)\in\psi(I)$, so by (classical) Borel closure,
    \begin{equation*}
        \psi(m)\cdot\frac{y_i^{w_j}}{y_j^{w_j}}\in \B{\psi(I)}.
    \end{equation*}
    Since $i<j$, we have $w_i\geq w_j$, so $y_i^{w_i-w_j}$ is a monomial. It follows that $\psi(m\frac{x_i}{x_j})\in\B{\psi(I)}$, so $m\frac{x_i}{x_j}\in I$, and we have shown that $I$ is strongly stable.
\end{proof}

The previous proposition shows that (for a fixed weight vector), the set of $w$-stable ideals is a subset of the set of strongly stable ideals. From this perspective, $w$-stable ideals are a refinement of strongly stable ideals. Contrast this with Dalzotto and Sbarra's definition of \textit{weighted strongly stable ideals} in \cite{dalzotto2005non} where, for any given monomial ideal, one can find a weight vector which makes the ideal weighted strongly stable. We can see the difference explicitly in the following example.

\begin{example}
    Let $I=(x_1^2,x_2^2)\subset S=\mathbb{K}[x_1,x_2]$. Using Dalzotto and Sbarra's definition (with order of variables reversed to make their notation match ours), $I$ is weighted strongly stable with respect to $w=(4,3)$. However, since $I$ is not strongly stable, there exists no weight vector for which $I$ is $w$-stable.
\end{example}

\begin{definition}
    Let $u,v\in S$ be monomials. We say that \textit{$u$ precedes $v$ in the $w$-Borel order} and write $u\succ_{w}v$ if $\psi(u)\succ\psi(v)$.
\end{definition}

\begin{proposition}
    Let $m,u\in S$ be monomials. Then $m\prec_w u$ if and only if $u\in\overline{m}$.
\end{proposition}

\begin{proof}
    By definition, $m\prec_w u$ if and only if $\psi(m)\prec \psi(u)$. From Lemma 2.11 of \cite{francisco2011borel}, we see that $\psi(m)\prec \psi(u)$ if and only if $\psi(u)\in Borel(\psi(m))$.
\end{proof}

Some authors (see \cite{gasharov2002resolutions}, \cite{herzog2002generic}) call $\overline{m}$ the \textit{shadow of $m$} or \textit{big shadow of $m$}. Rather than giving a combinatorial description of weighted Borel moves corresponding to classical Borel moves, we define a $w$-Borel move on $m$ as any move $m\mapsto u$ where $m\prec_w u$. The figure below gives a visualization in two variables.

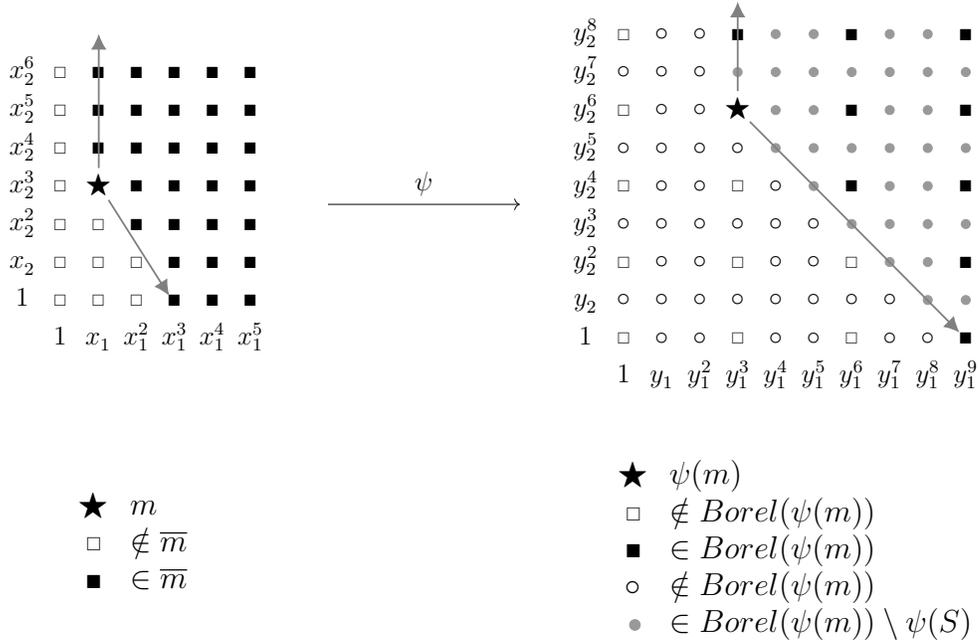
\begin{figure}[H]
    \caption{A visualization of $\psi:S\rightarrow R$ for $w=(3,2)$}
    \label{fig:psiexample}
    
    \begin{flushleft}
        \scalebox{0.85}{
        \begin{tikzpicture}
        \node (n)[] {
        \begin{ytableau}
        \none[x_2^6] & \none[\sqpt] & \none[\fsqpt] & \none[\fsqpt] & \none[\fsqpt] & \none[\fsqpt] & \none[\fsqpt]  \\
        \none[x_2^5] & \none[\sqpt] & \none[\fsqpt] & \none[\fsqpt] &\none[\fsqpt] & \none[\fsqpt] & \none[\fsqpt] \\
        \none[x_2^4] & \none[\sqpt] & \none[\fsqpt] & \none[\fsqpt] & \none[\fsqpt] & \none[\fsqpt] & \none[\fsqpt] \\
        \none[x_2^3] & \none[\sqpt] & \none[\bigstar] & \none[\fsqpt] & \none[\fsqpt] & \none[\fsqpt] & \none[\fsqpt]  \\
        \none[x_2^2] & \none[\sqpt] & \none[\sqpt] & \none[\fsqpt] & \none[\fsqpt] & \none[\fsqpt] & \none[\fsqpt] \\
        \none[x_2^{}] & \none[\sqpt] & \none[\sqpt] & \none[\sqpt] & \none[\fsqpt] & \none[\fsqpt] & \none[\fsqpt]\\
        \none[1] & \none[\sqpt] & \none[\sqpt] & \none[\sqpt] & \none[\fsqpt] & \none[\fsqpt] & \none[\fsqpt]\\
        \none[] & \none[1] & \none[x_1^{}] & \none[x_1^2] & \none[x_1^3] & \none[x_1^4] & \none[x_1^5]
        \end{ytableau}};
        \draw[-{Latex[length=7pt,width=7pt]},thick,gray] ([xshift=1.80cm,yshift=2.60cm]n.south west)--([xshift=2.75cm,yshift=1.1cm]n.south west);
        \draw[-{Latex[length=7pt,width=7pt]},thick,gray] ([xshift=1.647cm,yshift=3.1cm]n.south west)--([xshift=1.647cm,yshift=5.2cm]n.south west);
        \hspace{3cm}
        \draw[->] (0,0) -- (3,0) node[midway, above] {$\psi$};
        \hspace{7cm}

        \node (n2)[] {
        \begin{ytableau}
        \none[y_2^8] & \none[\sqpt] & \none[\circ] & \none[\circ] & \none[\fsqpt] & \none[\greycirc] & \none[\greycirc] & \none[\fsqpt] & \none[\greycirc] & \none[\greycirc]& \none[\fsqpt] \\
        \none[y_2^7] & \none[\circ] & \none[\circ] & \none[\circ] & \none[\greycirc] & \none[\greycirc] & \none[\greycirc] & \none[\greycirc] & \none[\greycirc]& \none[\greycirc]& \none[\greycirc] \\
        \none[y_2^6] & \none[\sqpt] & \none[\circ] & \none[\circ] & \none[\bigstar] & \none[\greycirc] & \none[\greycirc] & \none[\fsqpt] & \none[\greycirc] & \none[\greycirc] & \none[\fsqpt] \\
        \none[y_2^5] & \none[\circ] & \none[\circ] & \none[\circ] &\none[\circ] & \none[\greycirc] & \none[\greycirc] & \none[\greycirc] & \none[\greycirc]& \none[\greycirc]& \none[\greycirc]  \\
        \none[y_2^4] & \none[\sqpt] & \none[\circ] & \none[\circ] & \none[\sqpt] & \none[\circ] & \none[\greycirc] & \none[\fsqpt] & \none[\greycirc]& \none[\greycirc]& \none[\fsqpt]  \\
        \none[y_2^3] & \none[\circ] & \none[\circ] & \none[\circ] & \none[\circ] & \none[\circ] & \none[\circ] & \none[\greycirc] & \none[\greycirc]& \none[\greycirc] & \none[\greycirc] \\
        \none[y_2^2] & \none[\sqpt] & \none[\circ] & \none[\circ] & \none[\sqpt] & \none[\circ] & \none[\circ] & \none[\sqpt] & \none[\greycirc]& \none[\greycirc] & \none[\fsqpt] \\
        \none[y_2^{}] & \none[\circ] & \none[\circ] & \none[\circ] & \none[\circ] & \none[\circ] & \none[\circ] & \none[\circ] & \none[\circ]& \none[\greycirc]& \none[\greycirc] \\
        \none[1] & \none[\sqpt] & \none[\circ] & \none[\circ] & \none[\sqpt] & \none[\circ] & \none[\circ] & \none[\sqpt] & \none[\circ]& \none[\circ]& \none[\fsqpt]  \\
        \none[] & \none[1] & \none[y_1^{}] & \none[y_1^2] & \none[y_1^3] & \none[y_1^4] & \none[y_1^5] &\none[y_1^6] & \none[y_1^7] &\none[y_1^8] &\none[y_1^9]
        \end{ytableau}};
        \draw[-{Latex[length=7pt,width=7pt]},thick,gray] ([xshift=3.02cm,yshift=4.42cm]n2.south west)--([xshift=6.30cm,yshift=1.12cm]n2.south west);
        \draw[-{Latex[length=7pt,width=7pt]},thick,gray] ([xshift=2.83cm,yshift=4.9cm]n2.south west)--([xshift=2.83cm,yshift=6.3cm]n2.south west);
        \end{tikzpicture}}
    \end{flushleft}

\end{figure}

\begin{parcolumns}[colwidths={1=0.45\textwidth, 2=0.45\textwidth}]{2}
\colchunk{
  \hspace{1cm} 
  \begin{minipage}[t]{\linewidth}
    \vspace{3pt}
    \begin{itemize}
        \item[$\bigstar$\phantom{'}] $m$
        \item[\sqpt\phantom{2}] $\notin\overline{m}$
        \item[\fsqpt\phantom{2}] $\in\overline{m}$
    \end{itemize}
  \end{minipage}
}

\colchunk{
    \hspace{-1cm}
    \begin{minipage}[t]{\linewidth}
        \begin{itemize}
            \item[$\bigstar$\phantom{'}] $\psi(m)$
            \item[\sqpt\phantom{2}] $\notin Borel(\psi(m))$
            \item[\fsqpt\phantom{2}] $\in Borel(\psi(m))$
            \item[$\circ$\phantom{/}] $\notin Borel(\psi(m))$
            \item[$\greycirc$\phantom{2}] $\in Borel(\psi(m))\setminus\psi(S)$
        \end{itemize}
    \end{minipage}

}
\end{parcolumns}
\phantom{stuff\\}

In Figure \ref{fig:psiexample}, the monomial lattices for $S=\mathbb{K}[x_1,x_2]$ and $R=\mathbb{K}[y_1,y_2]$ are shown on the left and right respectively. We have $m=x_1x_2^3\in S$ (and its image $\psi(m)\in R$) represented by $\bigstar$. Monomials in $\overline{m}\subset S$ are represented by $\fsqpt$, while monomials in $S/\overline{m}$ are displayed as $\sqpt$. Monomials in $R\setminus\psi(S)$ are displayed with $\circ$ (and with $\greycirc$ if the monomial is in $\B{\psi(I)}$). Monomials in $\psi(S)\subset R$ are displayed as $\sqpt$ (and with $\fsqpt$ if the monomial is in $\B{\psi(I)}$).

Classically, Borel moves have been defined as those moves which move (in two variables) down and to the right. Rather than define $w$-Borel moves analagously, we will consider a $w$-Borel move as a move $m\mapsto u$ for any $u\in\overline{m}$.

\begin{definition}
    A \textit{$w$-Borel move} on a monomial $m\in S$ is an operation which sends $m$ to a monomial $u$ where $m\prec_{w} u$.
\end{definition}

While this definition lacks the explicit combinatorial description of classic Borel moves, we can state the analog of the definition of strongly stable ideals as a proposition below.

\begin{proposition}
    A monomial ideal $I$ is $w$-stable if and only if it is fixed by $w$-Borel moves.
\end{proposition}

\begin{proof}
    Assume $I\subset S$ is $w$-stable with $m\in I$ a monomial and $u\in S$ such that $u\succ_{w} m$. We need to show that $u\in I$. We have $\psi(u)\succ\psi(m)$, so $\psi(u)\in\B{\psi(I)}$. Thus, $u\in I$.

    Now, assume that $I\subset S$ is closed under $w$-Borel moves. We need to show that $I=\psi^{-1}(\B{\psi(I)})$. The forward inclusion is obvious. Let $v\in\psi^{-1}(\B{\psi(I)})$. Then $\psi(v)\in\B{\psi(I)}$. There must be some $m\in\psi(I)$ such that $\psi(v)\succ\psi(m)$. By definition, $v\succ_{w} m$. Since $I$ is closed under $w$-Borel moves, $v\in I$.
\end{proof}

\begin{proposition}[Proposition 2.12 of \cite{francisco2011borel}]\label{bg2.12}
    Every strongly stable ideal $J\subset R$ has a unique minimal set of Borel generators, $Bgens(J)$.
\end{proposition}

Strongly stable ideals can be described completely by their Borel generators. We have a natural generalization for $w$-stable ideals.

\begin{definition}
    Let $I\subset S$ be a $w$-stable ideal. Then we define the \textit{weighted Borel generators} of $I$ with respect to $w$ as the set
    \begin{equation*}
        Bgens_w(I) = \psi^{-1}(Bgens(\B{\psi(I)}))
    \end{equation*}
\end{definition}

Since every $w$-stable ideal $I\subset S$ corresponds to the strongly stable ideal $\B{\psi(I)}\subset R$, Proposition \ref{bg2.12} gives us the following.

\begin{proposition}
    For a fixed weight vector $w$, every $w$-stable ideal $I$ has a unique minimal set of weighted Borel generators $Bgens_w(I)$.
\end{proposition}

\begin{proposition}
    If $m,u\in S$ are monomials such that $m\prec u$, then $m\prec_w u$.
\end{proposition}

\begin{proof}
    We begin by writing both monomials in factored form (following notation in Definition \ref{def:borelorder})  $m=x_{i_1}\cdots x_{i_s}$, $u=x_{j_1}\cdots x_{j_r}$ and recall that $i_k\leq j_k$ for $1\leq k\leq s\leq r$. Now, consider $\psi(m),\psi(u)$ in factored form.
    \begin{alignat*}{4}
        &\hspace{1cm}w_{i_1} && \hspace{1cm}w_{i_2} && \hspace{1cm}w_{i_s} &&\\
        \psi(m) = &(\overbrace{y_{i_1}y_{i_1}\cdots y_{i_1}})&&(\overbrace{y_{i_2}y_{i_2}\cdots y_{i_2}})\cdots &&(\overbrace{y_{i_s}y_{i_s}\cdots y_{i_s}}) &&\\
        \psi(u) = &(\underbrace{y_{j_1}y_{j_1}\cdots y_{j_1}})&&(\underbrace{y_{j_2}y_{j_2}\cdots y_{j_2}})\cdots&&(\underbrace{y_{j_s}y_{j_s}\cdots y_{j_s}})\cdots &&(\underbrace{y_{j_r}y_{j_r}\cdots y_{j_r}})\\
        &\hspace{1cm}w_{j_1} && \hspace{1cm}w_{j_2} && \hspace{1cm}w_{j_s} &&\hspace{1cm}w_{j_r} \\
    \end{alignat*}
    Since $i_k\leq j_k$, we know that $w_{i_k}\geq w_{j_k}$. Thus, $\psi(m)\prec \psi(u)$, so $m\prec_w u$.
\end{proof}

\begin{proposition}
    If $I$ is $w$-stable for some weight vector $w$, then $Bgens_w(I)\subseteq Bgens(I)$.
\end{proposition}

\begin{proof}
    Suppose that $Bgens_w(I)$ contains some $b\in Bgens_w(I)\setminus Bgens(I)$. Then there exists some $m\in Bgens(I)$ such that $m\prec b$. Thus $\psi(m)\prec \psi(b)$ so $b\notin Bgens_w(I)$.
\end{proof}

Proposition 5.3 of \cite{ireland2023bijection} gives an algorithm for computing Borel generators which can easily be extended to compute weighted Borel generators. That algorithm is implemented in wStableIdeals as the borelGens method.

\begin{example}
Let $I=(x_1^2,x_1x_2^2,x_2^4)\subset S=\mathbb{K}[x_1,x_2]$ with weight vector $w=(2,1)$.\\
{\footnotesize
\begin{verbatim}

i7 : borelGens(ideal(x_1^2,x_1*x_2^2,x_2^4),Weights=>{2,1})

       4
o7 = {x }
       2

o7 : List


\end{verbatim}}
\end{example}

Many computations on strongly stable ideals can be reduced to computations on \textit{principal} strongly stable ideals as seen in the following three propositions from \cite{francisco2011borel}.

\begin{proposition}[Proposition 2.15 of \cite{francisco2011borel}]\label{bg2.15}
    Let $J_1=\B{T_1}$ and $J_2=\B{T_2}$ for sets of monomials $T_1,T_2\subset R$. Then
    \begin{equation*}
        J_1+J_2=\B{T_1\cup T_2}.
    \end{equation*}
\end{proposition}

\begin{definition}
    Let $u=y_{i_1}\cdots y_{i_r}$ and $v=y_{j_1}\cdots y_{j_s}$ be two monomials in factored form in $R$ with $s\leq r$. Put $l_t=min(i_t,j_t)$ with $l_t=s_t$ when $s<t\leq r$. Then $\mu=y_{t_1}\cdots y_{t_r}$ is the \textit{meet of $u$ and $v$ in the Borel order}, denoted $meet_{Borel}(u,v)=\mu$.
\end{definition}

\begin{proposition}[Proposition 2.16 of \cite{francisco2011borel}]\label{bg2.16}
    Let $u=y_{i_1}\cdots y_{i_r}$ and $v=y_{j_1}\cdots y_{j_s}$ with $r\geq s$. Set $k_q=min(i_q,j_q)$ with $k_q=i_q$ if $s< q\leq r$. Then set $t=y_{k_1}\cdots y_{k_r}$ be the meet of $u$ and $v$ in the Borel order. Then $\B{u}\cap \B{v}=\B{t}$.
\end{proposition}

\begin{proposition}[Proposition 2.17 of \cite{francisco2011borel}]\label{bg2.17}
    For monomials $u,v\in R$, $\B{u}\cdot\B{v}=\B{uv}$.
\end{proposition}

We can adapt each of the above propositions to the $w$-stable case, showing that computations on $w$-stable ideals can be reduced to computations on principal $w$-stable ideals.

\begin{proposition}\label{prop:basic1}
    Let $I_1=\overline{A_1}$ and $I_2=\overline{A_2}$ for sets of monomials $A_1,A_2\subset S$. Then $I_1+I_2=\overline{A_1\cup A_2}$.
\end{proposition}

\begin{proof}
    We begin by noticing that $I_1=\psi^{-1}(\B{\psi(A_1)})$ and $I_2=\psi^{-1}(\B{\psi(A_2)})$. Now,
    \begin{align*}
        I_1+I_2&=\psi^{-1}(\B{\psi(A_1)})+\psi^{-1}(\B{\psi(A_2)})\\
        &=\psi^{-1}(\B{\psi(A_1)}+\B{\psi(A_2)})\\
        &=\psi^{-1}(\B{\psi(A_1)\cup\psi(A_2)})\hspace{3cm}\text{(Proposition \ref{bg2.15})}\\
        &=\psi^{-1}(\B{\psi(A_1+A_2)})\\
        &=\overline{A_1+A_2}.
    \end{align*}
\end{proof}

We can consider the \textit{meet} of $u$ and $v$ in the $w$-Borel order which we will denote $meet_w(u,v):=\psi^{-1}(meet_{Borel}(\psi(u),\psi(v))$.

\begin{proposition}\label{prop:basic2}
    Let $u,v\in S$ be monomials with $q=meet_w(u,v)$. Then $\overline{u}\cap\overline{v}=\overline{q}$.
\end{proposition}

\begin{proof}
    Notice that 
    \begin{align*}
        \overline{u}\cap\overline{v} &= \psi^{-1}(\B{\psi(u)})\cap\psi^{-1}(\B{\psi(v)})\\
        &=\psi^{-1}(\B{\psi(u)}\cap\B{\psi(v)})\\
        &=\psi^{-1}(\B{\psi(q)}) \hspace{3cm}\text{(Proposition \ref{bg2.16})} \\
        &=\overline{q}.
    \end{align*}
\end{proof}

\begin{example}
    If $u=x_2^4$ and $v=x_1x_2$ in $S=\mathbb{K}[x_1,x_2]$ with $w=(2,1)$, then $\psi(u)=y_2^4$ and $\psi(v)=y_1^2y_2$. We have $meet(\psi(u),\psi(v))=y_1^2y_2^2$, so $meet(u,v)=x_1x_2^2$ in the $w$-Borel order.
\end{example}

\begin{proposition}\label{prop:basic3}
    For monomials $u,v\in S$, $\overline{u}\cdot\overline{v}=\overline{uv}$.
\end{proposition}

\begin{proof}
    We have $\overline{u}=\psi^{-1}(\B{\psi(u)})$ and similarly for $v$. Now, 
    \begin{align*}
        \overline{u}\cdot\overline{v} &= \psi^{-1}(\B{\psi(u)})\cdot\psi^{-1}(\B{\psi(v)})\\
        &= \psi^{-1}(\B{\psi(u)}\cdot\B{\psi(v)})\\
        &= \psi^{-1}(\B{(\psi(u)\cdot\psi(v))}) \hspace{3cm}\text{(Proposition \ref{bg2.17})} \\
        &= \psi^{-1}(\B{\psi(uv)})\\
        &= \overline{uv}.
    \end{align*}
\end{proof}

Given that we can reduce computations for arbitrary $w$-stable ideals to computations on principal $w$-stable ideals, we will be especially interested in principal $w$-stable ideals for the remainder of the paper.

\section{Stanley Decompositions}\label{sec:decomposition}

Following Francisco, Mermin, and Schweig in \cite{francisco2011borel}, we use truncations of strongly stable ideals to construct a Stanley decomposition. We begin this section by recalling the definition and notation of truncations and a useful lemma from \cite{francisco2011borel}.

\begin{definition}[Definition 3.10 of \cite{francisco2011borel}]
    Let $m\in R$ be a monomial with factorization $y_{i_1}\cdots y_{i_r}$ and let $d$ be a positive integer. If $d\leq deg(m)$, define the \textit{$d$-truncation of $m$}, denoted $trunc_d(m)$, to be $y_{i_1}\cdots y_{i_d}$. If $d>deg(m)$, then set $trunc_d(m)$ to be $m$ itself. For a monomial ideal $J\subset R$, define the \textit{$d$-truncation of $J$} to be the ideal $trunc_d(J)=(trunc_d(m):m\in J)$.
\end{definition}

\begin{lemma}[Lemma 3.12 of \cite{francisco2011borel}]\label{trunc_lemma}
    Let $J\subset R$ be a strongly stable ideal. Then
    \begin{equation*}
        trunc_d(J) = \B{trunc_d(m):m\in Bgens(J))}.
    \end{equation*}
\end{lemma}

In the classic setting, \textit{truncations} are used to obtain a Stanley decomposition of $R/J$ where $J$ is a strongly stable ideal. We obtain a similar result for $w$-stable ideals. First, we need a quick proposition.

\begin{proposition}\label{prop:complementsequal}
    If $I$ is $w$-stable, then $S/I=\psi^{-1}(R/\B{\psi(I)})$.
\end{proposition}

\begin{proof}
    If $u\in \psi^{-1}(R/\B{\psi(I)})$, then $\psi(u)\notin\B{\psi(I)}$. Clearly, $\psi(u)\notin\psi(I)$, so $u\notin I$.

    On the other hand, if $v\in S/I$, then $\psi(v)\notin\psi(I)$. Suppose, however, that $\psi(v)\in\B{\psi(I)}$. Then there exists some $m\in I$ such that $\psi(v)\succ\psi(m)$. Then $v\succ_w m$, but this contradicts $I$ being closed by $w$-Borel moves.
\end{proof}

\begin{theorem}
    Let $I\subset S$ be a $w$-stable ideal generated in degrees less than or equal to $d$. Let $G_s$ denote the degree $s$ generators of $trunc_s(\B{\psi(I)})$ which are not in $\B{\psi(I)}$. Then $S/I$ has the Stanley decomposition
    \begin{equation*}
        S/I = \bigoplus_{s=0}^{d-1} \Bigg(\bigoplus_{u\in \psi^{-1}(G_s)} u\cdot\mathbb{K}[x_j:\psi(u)y_j\notin trunc_{s+1}(\B{\psi(I)})]\Bigg).
    \end{equation*}
\end{theorem}

\begin{proof}
    We can start by getting a Stanley decomposition for $R/\B{\psi(I)}$
    \begin{equation*}
       R/\B{\psi(I)} = \bigoplus_{s=0}^{d-1} \Bigg(\bigoplus_{v\in G_s} u\cdot\mathbb{K}[y_j:vy_j\notin trunc_{s+1}(\B{\psi(I)})]\Bigg).
    \end{equation*}
    Since $S/I=\psi^{-1}(R/\B{\psi(I)})$,
    \begin{align*}
        S/I &= \psi^{-1}(R/\B{\psi(I)})\\
        &= \psi^{-1}\Bigg(\bigoplus_{s=0}^{d-1} \Bigg(\bigoplus_{v\in G_s} v\cdot\mathbb{K}[y_j:vy_j\notin trunc_{s+1}(\B{\psi(I)})]\Bigg)\Bigg)\\
        &= \bigoplus_{s=0}^{d-1} \Bigg(\bigoplus_{u\in \psi^{-1}(G_s)} u\cdot\psi^{-1}\bigg(\mathbb{K}[y_j:my_j\notin trunc_{s+1}(\B{\psi(I)})]\bigg)\Bigg)\\
        &= \bigoplus_{s=0}^{d-1} \Bigg(\bigoplus_{u\in \psi^{-1}(G_s)} u\cdot\mathbb{K}[x_j:my_j\notin trunc_{s+1}(\B{\psi(I)})]\Bigg).
    \end{align*}
The first line above is Proposition \ref{prop:complementsequal}. The second line uses the formula for Stanley decomposition of $R/\B{\psi(I)}$ found in Theorem 4.1 of \cite{francisco2011borel}. The third line follows from $\psi$ being an algebra homomorphism. To see that the fourth line follows, let $Z_y=\{y_{j_1},y_{j_2},\dots,y_{j_r}\}$ be any subset of the variables $y_1,\dots,y_n$. Then $\psi^{-1}(\mathbb{K}[Z_y])=\mathbb{K}[x_{j_1},x_{j_2},\dots,x_{j_r}]$.
\end{proof}

\begin{notation}
    Let $u$ be a monomial with factorization $u=y_{i_1}y_{i_2}\dots y_{i_k}$ and $i_1\leq i_2\leq \cdots \leq i_k$. We will use the notation $max(u)=i_k$ to denote the largest index of any variable dividing $u$.
\end{notation}

\begin{lemma}
    Let $m\in R$ be a monomial and assume that $u\in G(\B{trunc_{s}(v)})$. Then $uy_j\in\B{trunc_{s+1}(v)}$ if and only if $j\leq max(trunc_{s+1}(v))$.
\end{lemma}

\begin{proof}
    First, notice that $uy_j\in \B{trunc_{s+1}(v)}$ if and only if $trunc_{s+1}(v)\prec uy_j$. And this is only possible if $j\leq max(trunc_{s+1}(v))$.
\end{proof}

\begin{corollary}\label{principalSD}
    When $I=\overline{m}\subset S$ is a principal $w$-stable ideal, we get the following Stanley decomposition
\begin{equation*}
    S/\overline{m} = \bigoplus_{s=0}^{d-1} \Bigg(\bigoplus_{u\in \psi^{-1}(G_s)} u\cdot\mathbb{K}[x_j:j>max(trunc_{s+1}(\psi(m)))]\Bigg)
\end{equation*}
\end{corollary}

\begin{proof}
    Since $I=\overline{m}$ is a principal $w$-stable ideal, $\B{\psi(I)}=\B{\psi(m)}$ is a principal strongly stable ideal. By Lemma \ref{lem:truncgens}, we have $trunc_{s+1}(\B{\psi(I)})=\B{trunc_{s+1}(\psi(m))}$. Thus, the Stanley decomposition of $S/I$ is
    \begin{align*}
        S/I &= \bigoplus_{s=0}^{d-1} \Bigg(\bigoplus_{u\in \psi^{-1}(G_s)} u\cdot\mathbb{K}[x_j:\psi(u)y_j\notin trunc_{s+1}(\B{\psi(I)})]\Bigg)\\
        &= \bigoplus_{s=0}^{d-1} \Bigg(\bigoplus_{u\in \psi^{-1}(G_s)} u\cdot\mathbb{K}[x_j:\psi(u)y_j\notin \B{trunc_{s+1}(\psi(m))}]\Bigg)\\
        &= \bigoplus_{s=0}^{d-1} \Bigg(\bigoplus_{u\in \psi^{-1}(G_s)} u\cdot\mathbb{K}[x_j:j>max(trunc_{s+1}(\psi(m)))]\Bigg)
    \end{align*}
\end{proof}

\section{Hilbert Series}\label{sec:hilbseries}

If $I\subset S$ is a principal $w$-stable ideal and we consider $S$ as being the \textit{weighted} polynomial ring ($deg(x_i)=w_i$), then we get the following formula for the Hilbert series of $S/I$.

\begin{theorem}\label{thm:principalhs}
    Let $I=\overline{m}$ be a principal $w$-stable ideal generated by a degree $d$ monomial $m\in S$ and let $k_s=max(trunc_{s+1}(\psi(m)))$. Then
    \begin{equation*}
        HS(S/I) = \sum_{s=0}^{d-1}\frac{c_st^s}{\prod_{j=k_s+1}^n(1-t^{w_j})}
    \end{equation*}
    where $c_s$ is the cardinality of $\psi^{-1}(G_s)=\psi^{-1}(G(trunc_s(\B{\psi(I)}))_s\setminus\B{\psi(m)})$.
\end{theorem}

\begin{proof}
    By Corollary \ref{principalSD}, 
    \begin{align*}
        S/I &= \bigoplus_{s=0}^{d-1} \Bigg(\bigoplus_{u\in \psi^{-1}(G_s)} u\cdot\mathbb{K}[x_j:j>max(trunc_{s+1}(\psi(m)))]\Bigg),
    \end{align*}
    so we can write the Hilbert series of $S/I$ as
    \begin{align*}
        HS(S/I) &= HS\Bigg(\bigoplus_{s=0}^{d-1} \Bigg(\bigoplus_{u\in \psi^{-1}(G_s)} u\cdot\mathbb{K}[x_j:j>max(trunc_{s+1}(\psi(m)))]\Bigg)\Bigg)\\
        &= \sum_{s=0}^{d-1} \Bigg(\sum_{u\in\psi^{-1}(G_s)} HS\Big(u\cdot\mathbb{K}[x_j:j>k_s)]\Big)\Bigg)\Bigg)
    \end{align*}
    Each summand is of the form $u\cdot\mathbb{K}[x_j:j\geq k_s]$ for $u\in\psi^{-1}(G_s)$ and we can write its Hilbert series contribution as
    \begin{align*}
    HS(u\cdot\mathbb{K}[x_j:j\geq k_s])&=HS(u)\cdot HS(\mathbb{K}[x_j:j\geq k_s])\\
    &=t^{deg(u)}\cdot\frac{1}{\prod_{j=k_s+1}^n(1-t^{w_j})}\\
    &=\frac{t^s}{\prod_{j=k_s+1}^n(1-t^{w_j})}.
    \end{align*}
    Summing over all $u\in \psi^{-1}(G_s)$ multiplies the formula above by $c_s$, and summing over all degrees $0\leq s\leq d-1$ gives the desired formula.
\end{proof}

The coefficients $c_s$ in the formula above are tedious to compute since we need to compute $G(\B{trunc_s(m)})$ for $0\leq s\leq d-1$. In section \ref{sec:catalandiagrams}, we use \textit{Catalan diagrams} to more efficiently compute these coefficients. But first, we will take a closer look at the combinatorics of principal $w$-stable ideals and their connection to $n$-ary trees.

\section{Principal w-Stable Ideals are n-ary Trees}\label{sec:principaltrees}

\begin{definition}
    Let $m\in S$ be a monomial and fix a weight vector $w$. Define $\mathcal{T}_{w,m}^\infty$ to be the directed tree with base vertex $1$ and a directed edge $(v,vx_j)$ if $max(v)\leq j\leq max(trunc_{deg_w(v)+1}(\psi(m)))$. By convention, we set $max(1)=1$. We will also use $\mathcal{T}_{w,m}^d$ to denote the finite tree obtained by imposing the additional branching condition that $deg_w(v)<d$. When we suppress the exponent, we will take that to mean $\mathcal{T}_{w,m}=\mathcal{T}_{w,m}^{deg_w(m)}$.
\end{definition}

\begin{figure}[H]
\caption{$\mathcal{T}_{(1,1,1)}(x_3^3)$}
\label{fig:tree1}
\begin{center}
\newcommand\x{-6}
\newcommand\xa{-6}
\newcommand\xb{-6}
\newcommand\xc{-6}
\newcommand\y{0}
\newcommand\z{1.8}
\newcommand\drop{-2.5}
\scalebox{0.6}{
\begin{tikzpicture}[shorten >=1pt,->]
  \tikzstyle{vertex}=[circle,fill=black!15,minimum size=40pt,inner sep=2pt]
  \node[vertex] (1) at (0,0) {$1$};
  \node[vertex] (x) at (\xa,\drop)   {$x_1$};
  \node[vertex] (y) at (\y,\drop)  {$x_2$};
  \node[vertex] (z) at (\z,\drop) {$x_3$};
  \draw (1) -- (x);
  \draw (1) -- (y);
  \draw (1) -- (z);
  \node[vertex] (x2) at (\xa+\xb,2*\drop) {$x_1^2$};
  \node[vertex] (xy) at (\x + \y,2*\drop) {$x_1x_2$};
  \node[vertex] (xz) at (\x + \z, 2*\drop) {$x_1x_3$};
  \node[vertex] (y2) at (\y + \y, 2*\drop) {$x_2^2$};
  \node[vertex] (yz) at (\y + \z, 2*\drop) {$x_2x_3$};
  \node[vertex] (z2) at (\z + \z, 2*\drop) {$x_3^2$};
  \draw (x) -- (x2);
  \draw (x) -- (xy);
  \draw (x) -- (xz);
  \draw (y) -- (y2);
  \draw (y) -- (yz);
  \draw (z) -- (z2);
  \node[vertex] (x3) at (\xa+\xb+\xc,3*\drop) {$x_1^3$};
  \node[vertex] (x2y) at (2*\x+\y,3*\drop) {$x_1^2x_2$};
  \node[vertex] (x2z) at (2*\x+\z,3*\drop) {$x_1^2x_3$};
  \node[vertex] (y3) at (3*\y,3*\drop) {$x_2^3$};
  \node[vertex] (y2z) at (2*\y+\z,3*\drop) {$x_2^2x_3$};
  \node[vertex] (z3) at (3*\z,3*\drop) {$x_3^3$};
  \node[vertex] (xy2) at (\x+2*\y,3*\drop) {$x_1x_2^2$};
  \node[vertex] (xyz) at (\x+\y+\z,3*\drop) {$x_1x_2x_3$};
  \node[vertex] (xz2) at (\x+2*\z,3*\drop) {$x_1x_3^2$};
  \node[vertex] (yz2) at (\y+2*\z,3*\drop) {$x_2x_3^2$};
  \draw (x2) -- (x3);
  \draw (x2) -- (x2y);
  \draw (x2) -- (x2z);
  \draw (xy) -- (xy2);
  \draw (xy) -- (xyz);
  \draw (xz) -- (xz2);
  \draw (y2) -- (y3);
  \draw (y2) -- (y2z);
  \draw (yz) -- (yz2);
  \draw (z2) -- (z3);
\end{tikzpicture}
}
\end{center}
\end{figure}
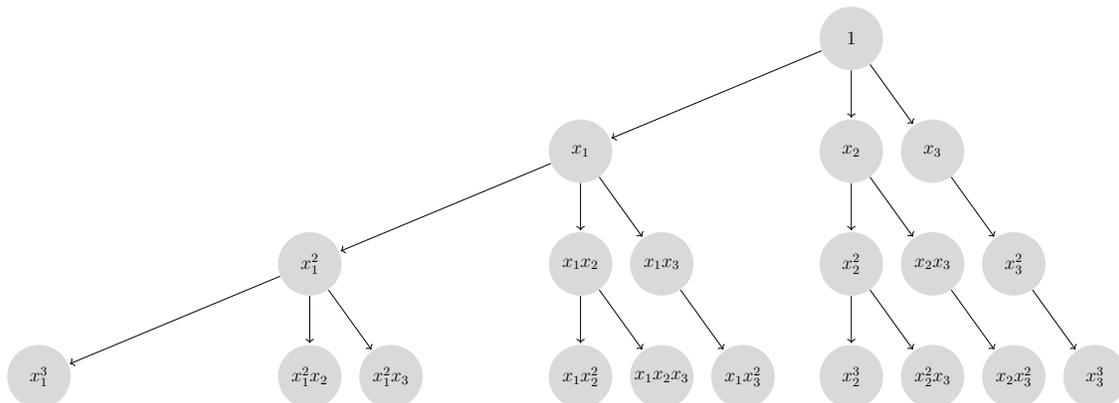

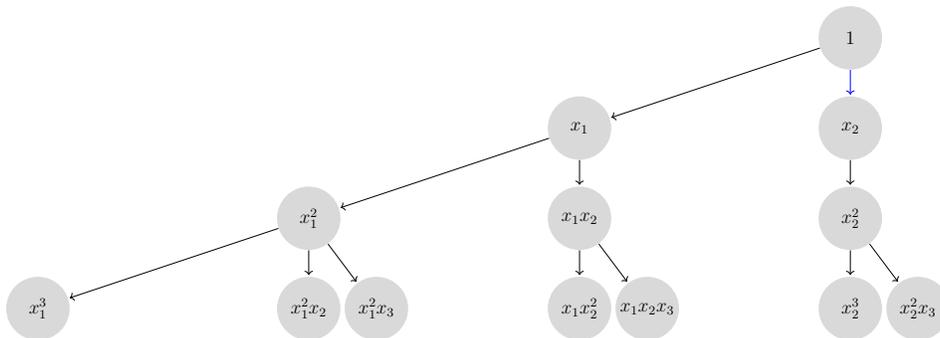
\begin{figure}[H]
\caption{$\mathcal{T}_{(1,1,1)}(x_2^2x_3)$}
\label{fig:tree2}
\begin{center}
\newcommand\x{-6}
\newcommand\y{0}
\newcommand\z{1.5}
\newcommand\drop{-2}
\newcommand\opac{0.3}
\scalebox{0.6}{
\begin{tikzpicture}[shorten >=1pt,->]
  \tikzstyle{vertex}=[circle,fill=black!15,minimum size=40pt,inner sep=2pt]
  \node[vertex] (1) at (0,0) {$1$};
  \node[vertex] (x) at (\x,\drop)   {$x_1$};
  \node[vertex] (y) at (\y,\drop)  {$x_2$};
  \draw (1) -- (x);
  \draw[draw=blue] (1) -- (y);
  \node[vertex] (x2) at (2*\x,2*\drop) {$x_1^2$};
  \node[vertex] (xy) at (\x + \y,2*\drop) {$x_1x_2$};
  \node[vertex] (y2) at (\y + \y, 2*\drop) {$x_2^2$};
  \draw (x) -- (x2);
  \draw (x) -- (xy);
  \draw (y) -- (y2);
  \node[vertex] (x3) at (3*\x,3*\drop) {};
    \node at (3*\x,3*\drop) {$x_1^3$};
  \node[vertex] (x2y) at (2*\x+\y,3*\drop) {};
    \node at (2*\x+\y,3*\drop) {$x_1^2x_2$};
  \node[vertex] (x2z) at (2*\x+\z,3*\drop) {};
    \node at (2*\x+\z,3*\drop) {$x_1^2x_3$};
  \node[vertex] (y3) at (3*\y,3*\drop) {};
  \node at (3*\y,3*\drop) {$x_2^3$};
  \node[vertex] (y2z) at (2*\y+\z,3*\drop) {};
  \node at (2*\y+\z,3*\drop) {$x_2^2x_3$}; 
  \node[vertex] (xy2) at (\x+2*\y,3*\drop) {};
  \node at (\x+2*\y,3*\drop) {$x_1x_2^2$};
  \node[vertex] (xyz) at (\x+\y+\z,3*\drop) {};
  \node at (\x+\y+\z,3*\drop) {$x_1x_2x_3$};
  \draw (x2) -- (x3);
  \draw (x2) -- (x2y);
  \draw (x2) -- (x2z);
  \draw (xy) -- (xy2);
  \draw (xy) -- (xyz);
  \draw (y2) -- (y3);
  \draw (y2) -- (y2z);
\end{tikzpicture}
}
\end{center}
\end{figure}

Note that $x_1x_3$ is not a vertex of $\mathcal{T}_{(1,1,1)}(x_2^2x_3)$, because $max(trunc_{1+1}(x_2^2x_3))=2$. Next, consider an example with a nontrivial weight vector.

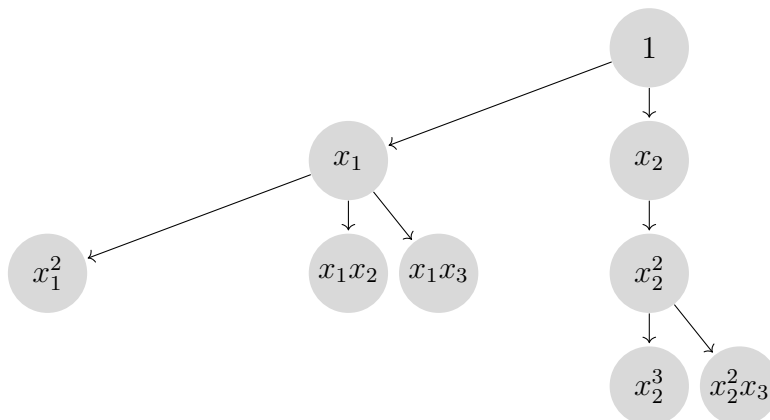
\begin{figure}[H]
\caption{$\mathcal{T}_{(4,2,1)}(x_2^2x_3)$}
\label{fig:tree3}
\begin{center}
\newcommand\x{-4}
\newcommand\y{0}
\newcommand\z{1.2}
\newcommand\drop{-1.5}
\newcommand\opac{0.3}
\scalebox{1}{
\begin{tikzpicture}[shorten >=1pt,->]
  \tikzstyle{vertex}=[circle,fill=black!15,minimum size=30pt,inner sep=2pt]
  \node[vertex] (1) at (0,0) {$1$};
  \node[vertex] (x) at (\x,\drop)   {$x_1$};
  \node[vertex] (y) at (\y,\drop)  {$x_2$};
  \draw (1) -- (x);
  \draw (1) -- (y);
  \node[vertex] (x2) at (2*\x,2*\drop) {};
  \node at (2*\x,2*\drop) {$x_1^2$};
  \node[vertex] (xy) at (\x + \y,2*\drop) {};
  \node at (\x+\y,2*\drop) {$x_1x_2$};
  \node[vertex] (xz) at (\x + \z, 2*\drop) {};
  \node at (\x+\z,2*\drop) {$x_1x_3$};
  \node[vertex] (y2) at (\y + \y, 2*\drop) {$x_2^2$};
  \draw (x) -- (x2);
  \draw (x) -- (xy);
  \draw (x) -- (xz);
  \draw (y) -- (y2);
  \node[vertex] (y3) at (3*\y,3*\drop) {};
  \node at (3*\y,3*\drop) {$x_2^3$};
  \node[vertex] (y2z) at (2*\y+\z,3*\drop) {};
  \node at (2*\y+\z,3*\drop) {$x_2^2x_3$}; 
  \draw (y2) -- (y3);
  \draw (y2) -- (y2z);
\end{tikzpicture}
}
\end{center}
\end{figure}

We can get $\mathcal{T}_w(m)$ in Macaulay2 by using the \verb|treeFromMonomial| method.

{\footnotesize
\begin{verbatim}


i8 : K[x_1,x_2,x_3]

o8 = QQ[x ..x ]
         1   3
         
o8 : PolynomialRing

i9 : treeFromMonomial(x_2^2*x_3,Weights=>{4,2,1})

o9 = Digraph{1 => {x , x }         }
                    1   2
                     2
             x  => {x , x x , x x }
              1      1   1 2   1 3
                     2
             x  => {x }
              2      2
             x x  => {}
              1 2
             x x  => {}
              1 3
              2
             x  => {}
              1
              2      3   2
             x  => {x , x x }
              2      2   2 3
              2
             x x  => {}
              2 3
              3
             x  => {}
              2
o9 : Digraph


\end{verbatim}}

The following theorem gives a dictionary between $\mathcal{T}_w(m)$ and $\overline{m}$.

\begin{theorem}
    Let $V_d$ denote the set of vertices of $\mathcal{T}_{w,m}^\infty$ whose degree is greater than or equal to $d$. Then $V_d=\psi^{-1}(\B{trunc_d(\psi(m))})$.
\end{theorem}

\begin{proof}
    The base case ($d=0$) is trivial.\\

    ($\subseteq$) Let $\tilde{v}\in V_{d+1}$ and write $\tilde{v}=vx_{j_1}x_{j_2}\cdots x_{j_s}\in V_{d+1}$ so that $max(v)\leq j_1\leq j_2\leq \cdots \leq j_s$, $deg_w(v)\leq d$, and $deg_w(vj_1)>d$. Then $v\in V_{d-w_{j_1}}$, so $\psi(v)\in\B{trunc_{d-w_{j_1}}(\psi(m))}$. Now the branch structure of $\mathcal{T}_{w,m}^\infty$ gives
    \begin{equation*}
        j_1\leq max(trunc_{d-w_{j_1}+1}(\psi(m))),
    \end{equation*}
    so $j_1\leq max(trunc_{d+1}(\psi(m)))$. Thus, $\psi(vx_{j_1})\in\B{trunc_{d+1}(\psi(m))}$. It follows that $\psi(\tilde{v})\in \B{trunc_{d+1}(\psi(m))}$.\\

    ($\supseteq$) Let $\tilde{u}\in\psi^{-1}(\B{trunc_{d+1}(\psi(m))})$ and write $\tilde{u}=ux_{j_1}x_{j_2}\cdots x_{j_s}$ so that $deg_w(u)\leq d$ and $deg_w(ux_{j_1})>d$. Then $\psi(ux_{j_1})=trunc_{deg_w(u)+w_{j_1}}(\nu)$ for some $\nu\in\B{\psi(m)}$. Thus, 
    \begin{align*}
        j_1 &= max(trunc_{deg_w(u)+w_{j_1}}(\nu))\\
        &= max(trunc_{deg_w(u)+1}(\nu))\\
        &\leq max(trunc_{deg_w(u)+1}(\psi(m)),
    \end{align*}
    so $ux_{j_1}\in V_{deg_w(u)+w_{j_1}}$. Now, we can repeat the procedure above for $j_2,\dots,j_s$ to get $\tilde{u}\in V_{deg_w(\tilde{u})}\subseteq V_{d+1}$.
\end{proof}

\begin{corollary}
    The set $sink(\mathcal{T}_w(m))$ is equal to $G(\overline{m})$.
\end{corollary}

\begin{proof}
    By the previous theorem, we know that $\overline{m}=V_{deg_w(m)}$. Since our tree is $\mathcal{T}_w(m)=\mathcal{T}_w^{deg_w(m)}(m)$, branches end when a vertex reaches $deg_w(m)$.
\end{proof}

We will now translate our earlier results on Stanley decompositions and Hilbert series for principal $w$-stable ideals in terms of $\mathcal{T}_w(m)$.

\begin{theorem}
    Let $V,E$ denote the sets of vertices and (directed) edges of $\mathcal{T}_w(m)$.
    The Stanley decomposition of $S/\overline{m}$ is given by 
    \begin{equation*}
        S/\overline{m} = \bigoplus_{u\in V\setminus sink(\mathcal{T}_w(m))} u\cdot\mathbb{K}[x_j:max(u)\leq j \text{ and } (u,ux_j)\notin E]
    \end{equation*}
\end{theorem}

\begin{proof}
    By Corollary \ref{principalSD}, we have
\begin{align*}
    S/\overline{m} &= \bigoplus_{s=0}^{d-1} \Bigg(\bigoplus_{u\in \psi^{-1}(G_s)} u\cdot\mathbb{K}[x_j:j>max(trunc_{s+1}(\psi(m))]\Bigg)
\end{align*}
Notice that 
\begin{align*}
    \psi^{-1}(G_s)&=\psi^{-1}(G(trunc_s(\B{\psi(m)}))\setminus\B{\psi(m)})\\
    &=\psi^{-1}(G(\B{trunc_s(\psi(m))}))\setminus\psi^{-1}(\B{\psi(m)})\\
    &=\psi^{-1}(G(\B{trunc_s(\psi(m))}))\setminus \overline{m}
\end{align*}
Now, if we sum over all $s=0,\dots,d-1$, we obtain all vertices except those in $\overline{m}$ which is exactly $sink(\mathcal{T}_w(m))$, so
\begin{align*}
    S/\overline{m} &= \bigoplus_{u\in V\setminus sink(\mathcal{T}_w(m))} u\cdot\mathbb{K}[x_j:max(u)\leq j \text{ and } (u,ux_j)\notin E]
\end{align*}
\end{proof}

We can also interpret the Hilbert series of $S/\overline{m}$ in terms of $\mathcal{T}_{w,m}$.

\begin{theorem}
    Let $V,E$ denote the sets of vertices and edges of $\mathcal{T}_{w,m}$. The Hilbert series of $S/\overline{m}$ is given by 
    \begin{equation*}
        \sum_{s=0}^{d-1}\frac{c_st^s}{\prod_{k_s+1}^n (1-t^{w_j})}
    \end{equation*}
    where $c_s$ is the number of degree $s$ vertices in $\mathcal{T}_{w,m}$ and $k_s$ is the largest index such that $(v,vx_{k_s})\in E$ for any degree $s$ vertex $v\in V$.
\end{theorem}

\section{Catalan Diagrams}\label{sec:catalandiagrams}

In this section, we generalize the notion of a Catalan diagram (see Definition 5.2 of \cite{francisco2011borel}) to a weighted Catalan diagram with respect to a given weight vector. The combinatorial structure of a principal $w$-stable ideal $\overline{m}$ allows us to efficiently track of the degree and maximal index of monomials in truncations of $\overline{m}$ using the weighted Catalan diagram $C_{w,m}$.

\begin{definition}
    For a fixed weight vector $w$ and monomial $m\in S$, we can form $C_{w,m}$, the \textit{weighted Catalan diagram} of $m$ with respect to $w$, as follows. Construct the classic Catalan diagram with shape $\psi(m)$ with an additional $max(w)-1$ rows of length $n$ at the bottom. Then fill in the entries with the following recursion.
    \begin{align*}
        C_{w,m}(0,1) &= 1\\
        C_{w,m}(a,b) &= \begin{cases} \sum_{k=1}^b C_{w,m}(a-w_b,k) & 0\leq a-w_b<d\text{  and  } max(trunc_{a-w_b+1}(\psi(m)))\geq b \\ \hspace{1.8cm} 0 & \text{otherwise} \end{cases}
    \end{align*}
\end{definition}

\begin{figure}[H]
\caption{The Catalan diagram $C_{(3,2,1)}(x_1x_2^3x_3^2)$}
\label{fig:cat1}
\begin{center}
\begin{ytableau}
    \none[s] & \none & \none & \none & \none[c_s] \\
    \none[0] & \none[1] & \none & \none & \none[1] \\
    \none[1] & 0  & \none & \none & \none[0] \\
    \none[2] & 0  & \none & \none & \none[0] \\
    \none[3] & 1  & \none & \none & \none[1] \\
    \none[4] & 0 & 0 & \none & \none[0] \\
    \none[5] & 0 & 1 & \none & \none[1] \\
    \none[6] & 1 & 0 & \none & \none[1] \\
    \none[7] & 0 & 1 & \none & \none[1] \\
    \none[8] & 0 & 1 & \none & \none[1] \\
    \none[9] & 1 & 1 & \none & \none[2] \\
    \none[10] & 0 & 1 & 2 & \none[3] \\
    \none[11] & *(lightgray)0 & *(lightgray)2 & *(lightgray)3 & \none[] \\
    \none[12] & *(lightgray)1 & *(lightgray)1 & *(lightgray)0 & \none[] \\
    \none[13] & *(lightgray)0 & *(lightgray)0 & *(lightgray)0 & \none[] \\
\end{ytableau}
\end{center}
\end{figure}

The rows of degree $s\geq deg(m)$ are shaded gray. The reason for this distinction will become clear in the next section. We can use the \verb|catalanDiagram| method to compute the Catalan diagram in Macaulay2.\\

{\footnotesize
\begin{verbatim}


i10 : catalanDiagram(x_1*x_2^3*x_3^2,Weights=>{3,2,1})

o10 =  | 1 0 0 |
       | 0 0 0 |
       | 0 0 0 |
       | 1 0 0 |
       | 0 0 0 |
       | 0 1 0 |
       | 1 0 0 |
       | 0 1 0 |
       | 0 1 0 |
       | 1 1 0 |
       | 0 1 2 |
       | 0 2 3 |
       | 1 1 0 |
       | 0 0 0 |

               14       3
o10 : Matrix ZZ   <-- ZZ
\end{verbatim}}

\begin{definition}
    Let $I\subset S$ be a $w$-stable ideal with minimal generators $G(I)=\{m_1,\dots,m_s\}$. For each $1\leq i\leq n$ and each degree $d$, let $q_i^d(I)$ be the cardinality of the set $\{m\in G(I):max(m)=i \text{ and } deg(m)=d\}$.
\end{definition}

\begin{theorem}\label{thm:catalan1}
Fix a weight vector $w$ and a monomial $m\in S$. Then for $1\leq a\leq d$,
    \begin{align*}
        C_{w,m}(a,b) &= q_b^a(\psi^{-1}(trunc_a(\B{\psi(m)})).
    \end{align*}
\end{theorem}

\begin{proof}
   We will prove this by induction on $1\leq a\leq d+max(w)-1$. First, notice that
    \begin{align*}
        C_{w,m}(1,b) &= \begin{cases} 1 & w_b=1 \\
        0 & w_b>1 \end{cases}
    \end{align*}
On the other hand, we have  $trunc_1(\B{\psi(m)}) = (y_1,y_2,\dots,y_{min(m)})$, so
\begin{equation*}
    \psi^{-1}(trunc_1(\B{\psi(m)}))=(x_1,x_2,\dots,x_{min(m)}).
\end{equation*}
Since $\psi^{-1}(y_i)=x_i$, it follows that
    \begin{align*}
        q_b^1(\psi^{-1}(trunc_1(\B{\psi(m)}))) &= \begin{cases} 1 & w_b=1 \\
        0 & w_b>1 \end{cases}
    \end{align*}
for $1\leq b\leq min(m)$. The base case is finished.\\

Now, assume that $C_{w,m}(l,k)=q_k^l(\psi^{-1}(trunc_l(\B{\psi(m)})))$ for $1\leq l< a\leq d$ and for all $k$. Any monomial in $G(\psi^{-1}(trunc_a(\B{\psi(m)})))$ must be of the form $ux_b=\psi^{-1}(\psi(u)y_b^{w_b})$ for some $u\in\psi^{-1}(trunc_{a-w_b}(\B{\psi(m)}))$ with $deg(u)=a-w_b$ and $max(u)\leq b$. Therefore,
\begin{align*}
    q_b^a(\psi^{-1}(trunc_a(\B{\psi(m)}))) &= \sum_{k=1}^b q_k^{a-w_b}(\psi^{-1}(trunc_{a-w_b}(\B{\psi(m)})))\\
    &= \sum_{k=1}^b C_{w,m}(a-w_b,k)\\
    &= C_{w,m}(a,b).
\end{align*}
\end{proof}

\begin{corollary}
    The coefficients $c_s$ for $0\leq s\leq d-1$ in Theorem 5 are given by the sum of the entries in the $s$\textsuperscript{th} row of $C_{w,m}$.
\end{corollary}

This gives us a combinatorial method for computing the Hilbert series of principal $w$-stable ideals. The following theorem will help us do something similar for Betti numbers in the next section.

\begin{lemma}\label{lem:truncgens}
    Fix a weight vector $w$ and a monomial $m\in S$. Then, for $d\leq deg(m)$, we have
    \begin{equation*}
        \psi^{-1}(G(trunc_d(\B{\psi(m)}))_d) = G(\psi^{-1}(trunc_d(\B{\psi(m)})))_d
    \end{equation*}
\end{lemma}

\begin{proof}
    ($\subseteq$) Let $u\in\psi^{-1}(G(trunc_d(\B{\psi(m)}))_d)$ so that $\psi(u)$ is a degree $d$ monomial generator of $trunc_d(\B{\psi(m)})$. Then $u\in\psi^{-1}(trunc_d(\B{\psi(m)}))$ also has degree $d$ and must be a generator, because every degree $d$ monomial of $\psi^{-1}(trunc_d(\B{\psi(m)}))$ is a generator.\\
    ($\supseteq$) Let $v\in G(\psi^{-1}(trunc_d(\B{\psi(m)})))_d$. Then $\psi(v)\in trunc_d(\B{\psi(m)})$ with degree $d$ and is therefore a generator. It follows that $v\in \psi^{-1}(G(trunc_d(\B{\psi(m)}))_d)$.
\end{proof}

\begin{theorem}\label{thm:catalan2}
    Fix a weight vector $w$ and a monomial $m\in S$ with $d=deg_w(m)$. For $d\leq a\leq d+max(w)-1$,
    \begin{equation*}
        C_{w,m}(a,b) = q_b^a(\overline{m}).
    \end{equation*}
\end{theorem}

\begin{proof}
    The base case is $a=d$, where it suffices to show that 
    \begin{equation*}
        G(\overline{m})_d=G(\psi^{-1}(trunc_d(\B{\psi(m)})))_d.
    \end{equation*}
    We have
    \begin{align*}
        u\in G(\overline{m})_d = G(\psi^{-1}(\B{\psi(m)})) &\iff \psi(u)\in G(\B{\psi(m)})_d \\
        &\iff \psi(u)\in G(trunc_d(\B{\psi(m)}))_d \\
        &\iff u\in \psi^{-1}(G(trunc_d(\B{\psi(m)}))_d)\\
        &\iff u\in G(\psi^{-1}(trunc_d(\B{\psi(m)})))_d,
    \end{align*}
    where the last line follows from Lemma \ref{lem:truncgens}. Since $G(\overline{m})_d=G(\psi^{-1}(trunc_d(\B{\psi(m)})))_d$, we have $q_b^d(\overline{m}) = q_b^d(\psi^{-1}(trunc_d(\B{\psi(m)})))$. From the previous theorem, we know that $C_{w,m}(d,b)=q_b^d(\psi^{-1}(trunc_d(\B{\psi(m)})))$, so we can conclude that $C_{w,m}(d,b) = q_b^d(\overline{m})$.\\
For the inductive step, we should assume that $C_{w,m}(l,b)=q_b^l(\overline{m})$ for $d\leq l<a<d+max(w)-1$. And we want to prove that $C_{w,m}(a,b)=q_b^{a}(\overline{m})$. We will consider two cases:
\begin{itemize}
    \item $a-w_b < d$
    \item $a-w_b \geq d$
\end{itemize}
In the case $a-w_b<d$, any degree $a$ generator is of the form $ux_b$ where
\begin{equation*}
    u\in \psi^{-1}(G(trunc_{a-w_b}(\B{\psi(m)}))_{a-w_b}).
\end{equation*}
Thus,
\begin{align*}
    q_b^a(\overline{m}) &= \sum_{k=1}^b q_k^{a-w_b}(\psi^{-1}(trunc_{a-w_b}(\B{\psi(m)})))\\
    &= \sum_{k=1}^b C_{w,m}(a-w_b,k)\\
    &= C_{w,m}(a,b).
\end{align*}

In the case $a-w_b\geq d$, we see that $C_{w,m}(a,b)=0$ by definition. Additionally, any monomial generator of $\psi^{-1}(trunc_{a}(\B{\psi(m)}))$ with maximum index $b$ must be of the form $vx_b$ for some degree $a-w_b$ generator $v\in G(\psi^{-1}(trunc_{a-w_b}(\B{\psi(m)})))$. But since $a-w_b\geq d$, we have $v\in\overline{m}$, so $vx_b\notin G(\overline{m})$.
\end{proof}

\section{Betti Numbers}\label{sec:betti}

In this section, we (again) consider $S$ to be the weighted polynomial ring with $deg(x_i)=w_i$. Since $w$-stable ideals are strongly stable, the Eliahou-Kervaire resolution gives the minimal free resolution. In particular, we can compute the total Betti numbers as follows.

\begin{theorem}[Eliahou-Kervaire]
    Let $I$ be a (strongly) stable ideal minimally generated by $G(I)=\{m_1,\dots,m_r\}$. Then
    \begin{equation}
        b_i(I)=\sum_{j=1}^r\binom{max(m_j)-1}{i-1}
    \end{equation}
\end{theorem}

We can see that the total Betti numbers $b_i(I)$ depend only on the maximum index of each generator. For a principal strongly stable ideal (or any strongly stable ideal generated entirely in degree $d$), we can rewrite this formula in terms of $q_i^d(I)$.

\begin{proposition}[Proposition 6.4 of \cite{francisco2011borel}]\label{prop:fmsbetti}
    Suppose that $I$ is a strongly stable ideal generated entirely in degree $d$. Then
    \begin{equation}\label{ekformula}
        b_i(I) = \sum_{j=1}^n \binom{j-1}{i-1}q_j^d(I)
    \end{equation}
\end{proposition}

\begin{example}
    Let $w=(1,1,1)$, $m=x_1x_2x_3^2$, and
    \begin{equation*}
        I=\overline{m}=(x_1x_2x_3^2,x_1^2x_3^2,x_1x_2^2x_3,x_1^2x_2x_3,x_1^3x_3,x_1x_2^3,x_1^2x_2^2,x_1^3x_2,x_1^4)
    \end{equation*}
    The Catalan diagram is shown in Figure \ref{fig:cat2} below.

\begin{figure}[H]
\caption{$C_{(1,1,1),x_1x_2x_3^2}$}
\label{fig:cat2}
\begin{center}
\begin{ytableau}
    \none[s] \\
    \none[0] & \none[1] \\
    \none[1] & 1 \\
    \none[2] & 1 & 1 \\
    \none[3] & 1 & 2 & 2 \\
    \none[4] & *(lightgray)1 & *(lightgray)3 & *(lightgray)5 \\
\end{ytableau}
\end{center}
\end{figure}
The last row of $C_{(1,1,1),x_1x_2x_3^2}$ tells us that we have $q_1^4(I)=1,q_2^4(I)=3,q_3^4(I)=5$. We can apply this information to Equation (\ref{ekformula}) to obtain the Betti numbers.
\begin{align*}
    b_1(I) &= 1\binom{0}{0} + 3\binom{1}{0} + 5\binom{2}{0} = 9\\
    b_2(I) &= 1\binom{0}{1} + 3\binom{1}{1} + 5\binom{2}{1} = 13\\
    b_3(I) &= 1\binom{0}{2} + 3\binom{1}{2} + 5\binom{2}{2} = 5
\end{align*}
We can verify this with the Betti table obtained by Macaulay2.
{\footnotesize
\begin{verbatim}


i11 : I = borelClosure(ideal(x_1*x_2*x_3^2))

                 2   2 2     2     2       3       3   2 2   3     4
o11 = ideal (x x x , x x , x x x , x x x , x x , x x , x x , x x , x )
             1 2 3   1 3   1 2 3   1 2 3   1 3   1 2   1 2   1 2   1

o11 : Ideal of QQ[x ..x ]
                   1   3

i12 : betti res I

             0 1  2 3
o12 = total: 1 9 13 5
          0: 1 .  . .
          1: . .  . .
          2: . .  . .
          3: . 9 13 5

o12 : BettiTally


\end{verbatim}}
\end{example}

The resolution in the previous example is linear (all nonzero Betti numbers are in the same row of Betti table) and $b_i(I)=b_{i,i+d}(I)$. This happens because principal strongly stable ideals are generated entirely in degree $d$. Since a principal $w$-stable ideal $I=\overline{m}$ can have generators of multiple degrees greater than or equal to $d=deg(m)$, we can't use Proposition \ref{prop:fmsbetti} for arbitrary $w$-stable ideals. We combine the Koszul complexes corresponding to the $q_j^d(I)$ to obtain a new formula (Theorem \ref{thm:poincareI}) for the Poincar\'e series of $w$-stable ideals. For a principal $w$-stable ideal, we can write the Poincar\'e series (and therefore the graded Betti numbers) in terms of $C_{w,m}$ (Corollary \ref{cor:principal_poincare}).

\begin{theorem}\label{thm:poincareI}
    Let $I\subset S$ be a $w$-stable ideal. Then the graded Poincar\'e series for $S/I$ is 
    \begin{equation}\label{betti_gen_fxn}
    P_I^S(u,t)=\sum_{m\in G(I)}ut^{deg(m)}\prod_{k=1}^{max(m)-1}(1+ut^{w_k})
    \end{equation}
\end{theorem}

\begin{proof}
    First, we will denote the set of subsets of $w$ truncated at $q$ with length $p$ by $A_{p,q}\subset\mathcal{P}(\{w_1,\dots,w_q\})$.

    Fix a monomial $m\in G(I)$. Let $d=deg(m)$ and $q=max(m)-1$. We will now describe the Koszul complex arising from $m$.
    The Eliahou-Kervaire resolution tells us that $\beta_{i,j}(I)$ counts the number of $S(-j)$ appearing in the $i$\textsuperscript{th} piece of the resolution. Since $m$ contributes the resolution of $(x_1,\dots,x_q)$ shifted by $d$, we have the Koszul complex below
    \begin{equation*}
        S(-d)\leftarrow \bigoplus_{a\in A_{1,q}} S(-sum(a)-d)\leftarrow \bigoplus_{a\in A_{2,q}} S(-sum(a)-d)\leftarrow\cdots\leftarrow \bigoplus_{a\in A_{q,q}} S(-sum(a)-d)
    \end{equation*}
    From this, we can see that the contribution to $\beta_{i,j}$ from $m$ is exactly the number of partitions of $j-d$ using distinct entries from $\{w_1,w_2,\dots,w_q\}$.
    The generating function for the number of partitions using distinct entries from $\{w_1,w_2,\dots,w_{q}\}$ is given by
    \begin{equation*}
        \prod_{k=1}^{q}(1+ut^{w_k}),
    \end{equation*}
    where the $u$ exponent corresponds to the number of summands and the $t$ exponent corresponds to the size of the partition $j-d$. In order to make the $t$ exponent correspond to $j$, we multiply the product by $t^d$. Similarly, since the $i$\textsuperscript{th} piece of the resolution corresponds to partitions with $i-1$ summands, we multiply the product by $u$.
    \begin{equation*}
        ut^d\prod_{k=1}^{q}(1+ut^{w_k}).
    \end{equation*}  
    The expression above is the generating function for the graded Betti numbers of $(x_1,\dots,x_q)$.
    Now, we can sum over $m\in G(I)$ as in the Eliahou-Kervaire formula and we're done.
\end{proof}

\begin{corollary}\label{cor:principal_poincare}
    If $I=\overline{m}$ is a principal $w$-stable ideal with $d=deg(m)$, then
    \begin{equation}
        P_I^S(t,u)=\sum_{a=d}^{d+max(w)-1}\sum_{b=1}^n C_{w,m}(a,b)\cdot ut^a\prod_{k=1}^{b-1}(1+ut^{w_k})
    \end{equation}
\end{corollary}

While the Poincar\'e series for a principal $w$-stable ideal is not as aesthetically pleasing as (\ref{betti_gen_fxn}), it allows us to compute $P_I^S(t,u)$ directly from the Catalan diagram.

\begin{example}
    Let $w=(3,2,1)$, $m=x_1x_2x_3^2$, and $I=\overline{m}=(x_1x_2x_3^2,x_1^2x_3,x_1x_2^2,x_1^2x_2,x_1^3)$. Notice that $deg(m)=7$ and $deg(x_1^2x_2)=8$ and $deg(x_1^3)=9$. Now, we build the Catalan diagram.

\begin{figure}[H]
\caption{$C_{(3,2,1),x_1x_2x_3^2}$}
\label{fig:cat3}
\begin{center}
\begin{ytableau}
    \none[s] \\
    \none[0] & \none[1] \\
    \none[1] & 0 \\
    \none[2] & 0  \\
    \none[3] & 1  \\
    \none[4] & 0 & 0 \\
    \none[5] & 0 & 1 \\
    \none[6] & 1 & 0 & 1 \\
    \none[7] & *(lightgray)0 & *(lightgray)1 & *(lightgray)2 \\
    \none[8] & *(lightgray)0 & *(lightgray)1 & *(lightgray)0 \\
    \none[9] & *(lightgray)1 & *(lightgray)0 & *(lightgray)0
\end{ytableau}
\end{center}
\end{figure}

To obtain the Catalan diagram above, we add $max(w_i)-1=3-1=2$ rows and fill them in as usual except we ignore any entries in rows $\geq d=deg(m)=7$. Then we can see both $deg(m_j)$ and $max(m_j)$ for each of the five generators directly from the extended Catalan diagram.

Now, we can use Corollary \ref{cor:principal_poincare} to compute the graded Poincar\'e series.
\begin{align*}
    P_I^S(t,u)&=\sum_{a=7}^{9} \sum_{b=1}^3 C_{w,m}(a,b)\cdot ut^a\prod_{k=1}^{b-1}(1+ut^{w_k})\\
    &=ut^7(1+ut^3) + 2ut^7(1+ut^3)(1+ut^2) + ut^8(1+ut^3) + ut^9\\
    &=3ut^7 + ut^8 + ut^9 + 2u^2t^9 + 3u^2t^{10} + u^2t^{11} + 2u^3t^{12}
\end{align*}

The \verb|poincareSeries method| from \verb|wStableIdeals| allows us to compute the Poincar\'e series of $I$, which we can compare with the Betti table above.
{\footnotesize
\begin{verbatim}


i13 : poincareSeries(x_1*x_2*x_3^2,Weights=>{3,2,1})

        12 3    11 2     10 2     9 2    9     8      7
o13 = 2t  u  + t  u  + 3t  u  + 2t u  + t u + t u + 3t u

o13 : ZZ[t..u]


\end{verbatim}}
\end{example}

We can check that this agrees with the Betti table.
{\footnotesize
\begin{verbatim}


i14 : w = {3,2,1}

o14 = {3, 2, 1}

o14 : List

i15 : S = K[x_1,x_2,x_3,Degrees=>w]

o15 = S

o15 : PolynomialRing

i16 : I = borelClosure(ideal(x_1*x_2*x_3^2),Weights=>w)

                  2   2       2   2     3
o16 = ideal (x x x , x x , x x , x x , x )
              1 2 3   1 3   1 2   1 2   1

o16 : Ideal of S

i17 : betti res I

              0 1 2 3
o17 = total:  1 5 6 2
           0: 1 . . .
           1: . . . .
           2: . . . .
           3: . . . .
           4: . . . .
           5: . . . .
           6: . 3 . .
           7: . 1 2 .
           8: . 1 3 .
           9: . . 1  2

o17 : BettiTally


\end{verbatim}}

If we fix a  that different choices of weight vector produce different Betti tables, but the total Betti numbers $b_i(I)$ are unaffected.
{\footnotesize
\begin{verbatim}


i18 : w = {1,1,1}

o18 = {1, 1, 1}

o18 : List

i19 : S = K[x_1,x_2,x_3,Degrees=>w]

o19 = S

o19 : PolynomialRing

i20 : I = borelClosure(ideal(x_1*x_2*x_3^2),Weights=>{3,2,1})

                  2   2       2   2     3
o20 = ideal (x x x , x x , x x , x x , x )
              1 2 3   1 3   1 2   1 2   1

o20 : Ideal of S

i21 : betti res I

             0 1 2 3
o21 = total: 1 5 6 2
          0: 1 . . .
          1: . . . .
          2: . 4 4 1
          3: . 1 2 1

o21 : BettiTally


\end{verbatim}}

Since the total Betti numbers do not depend on the weight vector, a good choice of weight vector $w$ (where $I$ is $w$-stable and the cardinality of $Bgens_w(I)$ is small relative to the cardinality of $Bgens(I)$) reduces the number of computations required to obtain the (total) Betti numbers.\\

\section{Principal Cones}\label{sec:principalcone}

\begin{question}
    Fix a strongly stable ideal $I$. For which weight vectors (if any) is $I$ principally $w$-stable?
\end{question}

This section is dedicated to exploring the question above. Theorem \ref{principalConeThm} provides an answer to this question in the form of an algorithm. This algorithm produces the space of weight vectors for which an Artinian strongly stable ideal is principally $w$-stable. 

Using the language of $n$-ary trees, we can ask if a given tree $\mathcal{T}$ is equal to $\mathcal{T}_{w,m}$ for some $m$ and $w$. It's clear what $m$ should be. It must be the smallest $lex$ vertex in $sink(\mathcal{T})$. Determining $w$ is a more difficult task and we devote the rest of the section to it.

\begin{definition}
    Let $I$ be a strongly stable ideal. Define $\mathcal{T}_I$ as the tree with base vertex $1$ and a directed edge $(v,vx_j)$ if $vx_j=trunc_{deg_{\mathbb{1}}(vx_j)}(g)$ for some $g\in G(I)$ with $g\neq v$.
\end{definition}

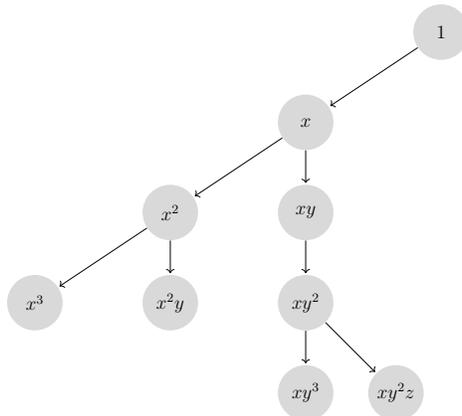
\begin{figure}[H]
\caption{$\mathcal{T}_I$ for $I=(x^3,x^2y,xy^3,xy^2z)$}
\label{fig:tree4}
\begin{center}
\newcommand\x{-3}
\newcommand\y{0}
\newcommand\z{2}
\newcommand\drop{-2}
\newcommand\opac{0.3}
\scalebox{0.6}{
\begin{tikzpicture}[shorten >=1pt,->]
  \tikzstyle{vertex}=[circle,fill=black!15,minimum size=35pt,inner sep=2pt]
  \node[vertex] (1) at (0,0) {$1$};
  \node[vertex] (x) at (\x,\drop)   {$x$};
  \draw (1) -- (x);
  \node[vertex] (x2) at (2*\x,2*\drop) {$x^2$};
  \node[vertex] (xy) at (\x+\y,2*\drop) {$xy$};
  \draw (x) -- (x2);
  \draw (x) -- (xy);
  \node[vertex] (x3) at (3*\x,3*\drop) {$x^3$};
  \node[vertex] (x2y) at (2*\x+\y,3*\drop) {$x^2y$};
  \node[vertex] (xy2) at (\x+2*\y,3*\drop) {$xy^2$};
  \draw (x2) -- (x3);
  \draw (x2) -- (x2y);
  \draw (xy) -- (xy2);
  \node[vertex] (xy3) at (\x+3*\y,4*\drop) {$xy^3$};
  \node[vertex] (xy2z) at (\x+2*\y+\z,4*\drop) {$xy^2z$};
  \draw (xy2) -- (xy3);
  \draw (xy2) -- (xy2z);
\end{tikzpicture}
}
\end{center}
\end{figure}

We can use $\mathcal{T}_I$ to obtain the space of weight vectors for which $I$ is principally $w$-stable. First, we introduce some notation.\\

\begin{lemma}\label{lem:tau}
    Let $m=\mathbf{x}^{(a_1,\dots,a_n)}$ and $v=\mathbf{x}^{(b_1,\dots,b_n)}$ be monomials. For any $1\leq k\leq n$, $k=max(trunc_{deg_w(v)+1}(\psi(m)))$ if and only if both of the following hold.
    \begin{align*}
        \sum_{i=1}^{k-1} w_ia_i &\leq \sum_{i=1}^n w_ib_i\\
        \sum_{i=1}^{k} w_ia_i &> \sum_{i=1}^n w_ib_i\\
    \end{align*}
\end{lemma}

\begin{proof}
    First, we get a condition to ensure that $k$ is small enough.
    \begin{align*}
        k\leq max(trunc_{deg_w(v)+1}(\psi(m))) &\iff \sum_{i=1}^{k-1} w_ia_i < \Big(\sum_{i=1}^n w_ib_i\Big) + 1\\
        &\iff \sum_{i=1}^{k-1} w_ia_i \leq \sum_{i=1}^n w_ib_i.
    \end{align*}
    We get a second condition when we force $k$ to be large enough.
    \begin{align*}
        k\geq max(trunc_{deg_w(v)+1}(\psi(m))) &\iff \sum_{i=1}^{k} w_ia_i \geq \Big(\sum_{i=1}^n w_ib_i\Big) + 1\\
        &\iff \sum_{i=1}^{k} w_ia_i > \sum_{i=1}^n w_ib_i.
    \end{align*}

\end{proof}

When $k$ is the largest index such that $(v,vx_k)$ is an edge of $\mathcal{T}_I$, denote the space of weight vectors satisfying the two inequalities in Lemma \ref{lem:tau} by $\tau_{m,v}$. Intuitively, $\tau_{m,v}$ is the space of weight vectors for which the branching structure at $v$ is compatible with $\mathcal{T}_{w,m}$.

Let $u=\mathbf{x}^\alpha,v=\mathbf{x}^\beta\in S$ be monomials. Then let $\sigma_{u,v}$ denote the half-space defined by
\begin{equation*}
    \sum_{i=1}^n w_i(b_i-a_i)\geq 0.
\end{equation*}
Similarly, let $\rho_{u,v}$ denote the half-space
\begin{equation*}
    \sum_{i=1}^n w_i(b_i-a_i)<0.
\end{equation*}
Intuitively, $\sigma_{m,v}$ is the half-space of weight vectors for which $deg_w(m)\leq deg_w(v)$ and $\rho_{m,v}$ is the space for which $deg_w(m)>deg_w(v)$.

\begin{theorem}\label{principalConeThm}
    Let $I$ be a strongly stable ideal with $m\in G(I)$ the smallest $lex$ generator of $I$. Let $V$ be the vertex set of $\mathcal{T}_I$. For each $u\in V\setminus sink(\mathcal{T_I})$, let $k_u$ denote the largest variable branching from $u$. Denote the set of vertices which share an edge with a vertex in $sink(\mathcal{T}_I)$ by $subsink(\mathcal{T}_I)$. Then $I$ is principally $w$-stable if and only if
    \begin{equation*}
        w\in\mathcal{P}_I = \Bigg(\bigcap_{v\in sink(\mathcal{T}_I)} \sigma_{m,v}\Bigg )\cap \Bigg(\bigcap_{u\in subsink(\mathcal{T}_I)} \rho_{m,v}\Bigg ) \cap \Bigg (\bigcap_{u\in V\setminus sink(\mathcal{T}_I)} \tau_{m,u,k_u}\Bigg ).
    \end{equation*}
\end{theorem}

\begin{proof}
     The first intersection makes sure that degrees of sinks are larger than or equal to $m$. The second intersection ensures that the degrees of subsinks are less than the degree of $m$, and the third intersection guarantees that the branching structure of $\mathcal{T}_I$ and $\mathcal{T}_{w,m}$ agree at each branching vertex.
\end{proof}

We will refer to $\mathcal{P}_I$ as the \textit{principal region} of $I$ and its closure as the \textit{principal cone}. Any lattice point in the interior of the principal cone gives a weight vector $w$ for which $I=\overline{m}$ where $m$ is the smallest $lex$ monomial in $G(I)$. The \verb|principalCone| method allows us to compute the principal cone of a strongly stable ideal in Macaulay2.
{\footnotesize
\begin{verbatim}


i20 : S = K[x,y,z]

o20 = S

o20 : PolynomialRing

i21 : I = ideal(x^3,x^2*y,x*y^3,x*y^2*z)

              3   2      3     2
o21 = ideal (x , x y, x*y , x*y z)

o21 : Ideal of S

i22 : P = principalCone I

o22 = Cone{...1...}

o22 : Cone

i22 : rays P

o22 = | 1 2 2 |
      | 1 1 1 |
      | 0 0 1 |
     
              3       3
o22 : Matrix ZZ  <-- ZZ

i23 : principalWeightVector I

o23 = | 5 |
      | 3 |
      | 1 |
     
              3       1
o23 : Matrix ZZ  <-- ZZ


\end{verbatim}}

\section{Applications}\label{sec:apps}

 In \cite{onn1999cutting}, Onn and Sturmfels describe \textit{corner cuts} as subsets of nonnegative integer vectors which are linearly separable from their complement. Theorems 2.0 and 5.0 of \cite{onn1999cutting} show that vanishing ideals of generic configurations of $k$ points in affine $n$-space have corner cut initial ideals. This is one class of examples which seem to be well-suited for computations as $w$-stable ideals, as illustrated below.

\begin{table}[h]
    \centering
    \begin{tabular}{|c|c|c|c|c|}
        \hline
        \textbf{\phantom{Cone}} & \textbf{$w$} & \textbf{$I$} & \textbf{$Bgens(I)$} & \textbf{$Bgens_w(I)$} \\
        \hline
        $a$ & $(8,8,1)$ & $(x,y,z^8)$ & $\{y,z^8\}$ & \{$z^8$\}\\
        \hline
        $b$ &  $(7,6,1)$ & $(x,y^2,yz,z^7)$ & $\{x,yz,z^7\}$ & $\{z^7\}$ \\
        \hline
        $c$ & $(5,5,1)$ & $(x^2,xy,y^2,xz,yz,z^6)$ & $\{yz,z^6\}$ & $\{z^6\}$ \\
        \hline
        $d$ & $(6,4,1)$ & $(x,y^2,yz^2,z^6)$ & $\{x,y^2,yz^2,z^6\}$ & $\{z^6\}$ \\
        \hline
        $e$ & $(4,3,1)$ & $(x^2,xy,y^2,xz,yz^2,z^5)$ & $\{xz,y^2,yz^2,z^5\}$ & $\{z^5\}$ \\
        \hline
        $f$ & $(10,5,2)$ & $(x,y^2,yz^3,z^5)$ & $\{x,y^2,yz^3,z^5\}$ & $\{z^5\}$ \\
        \hline
        $g$ & $(8,3,2)$ & $(x,y^3,y^2z,yz^3,z^4)$ & $\{x,y^2z,z^4\}$ & $\{z^4\}$ \\
        \hline
        $h$ & $(5,3,2)$ & $(x^2,xy,xz,y^3,y^2z,yz^2,z^4)$ & $\{xz,yz^2,z^4\}$ & $\{yz^2,z^4\}$ \\
        \hline
        $i$ & $(2,2,1)$ & $(x^2,xy,y^2,xz^2,yz^2,z^4)$ & $\{y^2,yz^2,z^4\}$ & $\{z^4\}$ \\
        \hline
    \end{tabular}
    \caption{Weight vectors and generating sets for initial ideals in Figure \ref{fig:ginfan222}}
    \label{tab:ginfan222}
\end{table}

\begin{figure}[h]
    \centering
    \caption{Gr\"obner fan of generic non-homogeneous ideal of type $(2,2,2)$}
    \label{fig:ginfan222}
    \scalebox{0.9}{
    \begin{tikzpicture}
        \node[anchor=south west,inner sep=0] (image) at (0,0) {\includegraphics[width=0.9\linewidth]{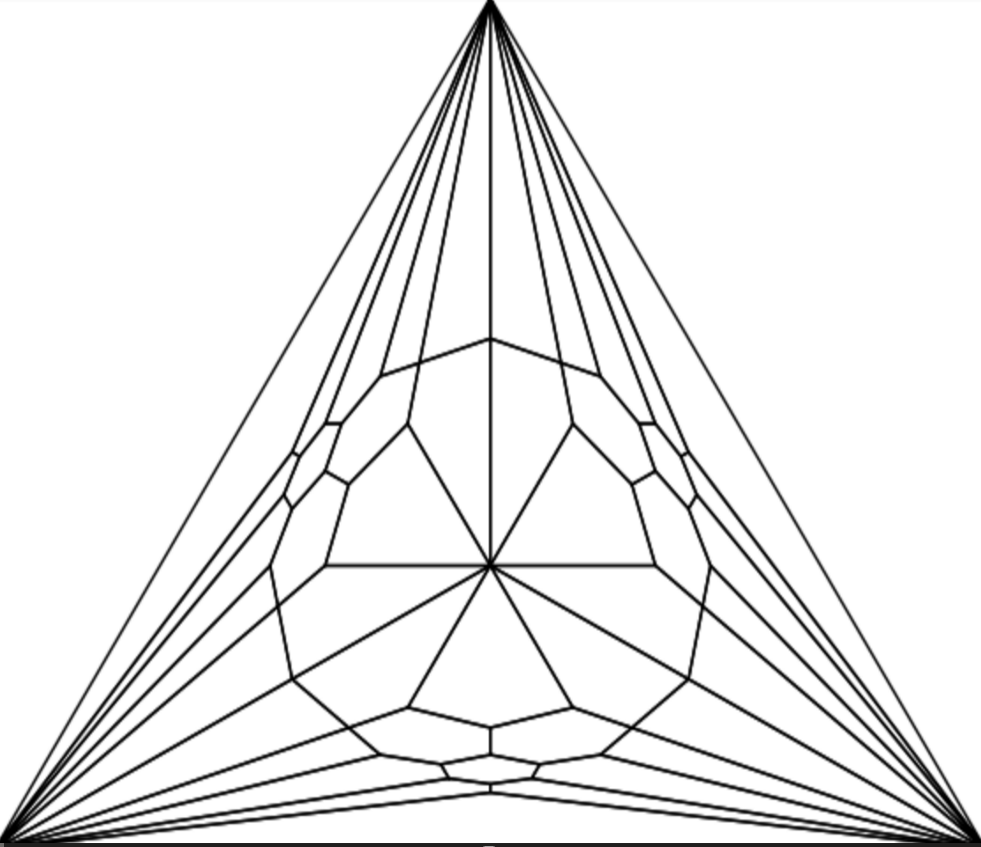}};
        \begin{scope}[x={(image.south east)},y={(image.north west)}]
            \node at (-0.04,0) [color=black] {(1,0,0)};
            \node at (1.04,0) [color=black] {(0,1,0)};
            \node at (0.5,1.02) [color=black] {(0,0,1)};
            \node at (0.5,0.03) {$a$};
            \node at (0.5,0.093) {$c$};
            \node at (0.435,0.13) {$e$};
            \node at (0.38,0.088) {$d$};
            \node at (0.343,0.115) {$f$};
            \node at (0.28,0.14) {$g$};
            \node at (0.38,0.2) {$h$};
            \node at (0.5,0.2) {$i$};
        \end{scope}
    \end{tikzpicture}}
\end{figure}
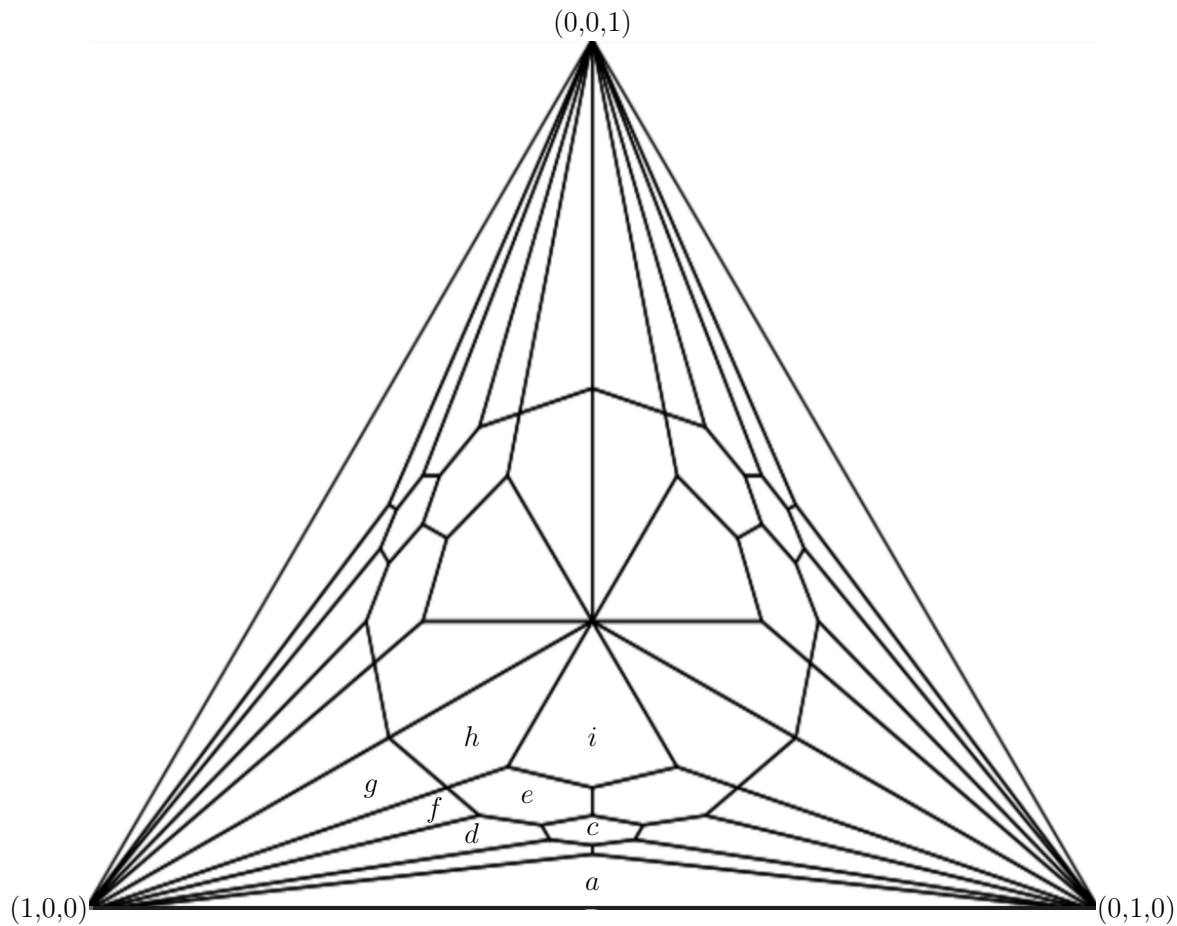

Figure \ref{fig:ginfan222} shows the (Schlegel diagram of the) Gr\"obner fan of $J=(F_1,F_2,F_3)$ where the $F_i$ are generic, non-homogeneous polynomials of degree 2. The 9 maximal cones with interior points contained in the strict interior of the region where $x>y>z$ are labelled $a,\dots,i$. Note that $b$ does not appear in the figure but is used to denote the maximal cone adjacent to cones $a,c,$ and $d$. One can check that each initial ideal is $w$-stable with respect to the corresponding weight vector in the interior of its maximal cone. This motivates the following.

Table \ref{tab:ginfan222} shows a weight vector $w$ in the interior, the ideal $I=in_<(J)$, $Bgens(I)$, and $Bgens_w(I)$ for each maximal cone. In all cases except $h$, the initial ideal $I$ can be expressed as a principal $w$-stable ideal by using the corresponding weight vector in the first column. Clearly, an Artinian principal $w$-stable ideal is a corner cut. However, the next example shows that no weight vector exists for which the corner cut $(x^2,xy,xz,y^3,y^2z,yz^2,z^4)$ is principally $w$-stable.

\begin{example}\label{counter}
    The ideal $I=(x^2,xy,xz,y^3,y^2z,yz^2,z^4)\subset\mathbb{K}[x,y,z]$ is a corner cut, because it is the initial ideal of a generic, non-homogeneous ideal (corresponding to $h$ in Table \ref{tab:ginfan222}). Since $z^4$ cannot generate $yz^2$ without also generating $y^2$, there exists no weight vector $w$ such that $I=\overline{z^4}$.
\end{example}

\begin{question}
    Assume $I\subset\mathbb{K}[x_1,\dots,x_n]$ is a corner cut. 
    Is there an upper bound on the minimal number of weighted Borel generators (for a suitable choice of weight vector) for $I$?\\
\end{question}

Recall that a monomial order $<$ on $S=\mathbb{K}[x_1,\dots,x_n]$ is a \textit{total} order on the set of monomials which respects multiplication ($u\leq v\implies ut\leq vt$ for any monomials $t,u,v$). For a given monomial order $<$ and a fixed ideal $I\subset S$, one can obtain the monomial ideal $in_<(I)$. A weight vector $w=(w_1,\dots,w_n)$ gives a \textit{partial} order on the set of monomials. Note that the initial ideal of $I$ with respect to $w$, denoted $in_w(I)$ might not be a monomial ideal. For background on initial ideals and their relationship to the Gr\"obner fan, see \cite{fkj_computinggrobfans}.

\begin{conjecture}\label{conj1}
    Fix a generic non-homogeneous ideal $I$ and a monomial order $<$. If $w\in\mathbb{N}^n_{>0}$ is a weight vector and $in_w(I)=in_<(I)$, then $in_<(I)$ is $w$-stable.
\end{conjecture}

\section{Conclusion}

Restrictions on the structure of monomial ideals can reduce the number of monomial generators needed to describe a fixed monomial ideal $I$. Francisco, Mermin, and Schweig show how to use the structure of strongly stable ideals to compress the generating set from the classic monomial generators $G(I)$ to a subset $Bgens(I)$. In this paper, we showed how a choice of weight vector $w$ restricts the set of strongly stable ideals to a subset of monomial ideals called $w$-stable ideals. This restriction allows us to describe $w$-stable ideals completely in terms of their weighted Borel generators $Bgens_w(I)$. In summary, for a $w$-stable ideal $I$, we have
\begin{equation*}
    Bgens_w(I)\subseteq Bgens(I)\subseteq G(I)
\end{equation*}

In Section \ref{subsec:wstablebasics}, we showed that computations on $w$-stable ideals can be decomposed into computations on principal $w$-stable ideals whose combinatorial structure was examined in Section \ref{sec:principaltrees}. Formulas for the Stanley decomposition, Hilbert series, and Betti numbers of principal $w$-Stable ideals were developed in Sections \ref{sec:decomposition}, \ref{sec:hilbseries}, and \ref{sec:betti}. In Section \ref{sec:principalcone}, we showed how to compute the space of weight vectors for which a given strongly stable ideal is principally $w$-stable. In Section \ref{sec:apps}, we explored generic, non-homogeneous initial ideals and presented evidence that using a weight vector $w$ corresponding to the given monomial order produces a $w$-stable ideal (Conjecture \ref{conj1}).

\newpage
\printbibliography
\end{document}